\documentclass[11pt]{amsart}

\usepackage{amsmath,amssymb,epsfig}
\usepackage{graphicx}
\usepackage{fullpage}
\usepackage{amscd}
\usepackage{amsfonts, amssymb} %,
\usepackage{stmaryrd}
\usepackage[all]{xy}
\usepackage{upgreek}
\usepackage{datetime2} 
\usepackage{latexsym}
\usepackage{cite}
\usepackage[usenames]{color}
\usepackage{xcolor}
\usepackage{url}
\usepackage{boxedminipage}
\usepackage{amsbsy}
\usepackage{framed}
\usepackage[colorlinks=false, linktocpage=true]{hyperref}

\usepackage{enumitem}
\usepackage{bm}

\newcommand{\fg}{\mathfrak g}
\newcommand{\frs}{\mathfrak s}

\newcommand{\cA}{{\mathcal A}}

\newcommand{\cI}{{\mathcal I}}

\newcommand{\cE}{{\mathcal E}}

\newcommand{\cP}{{\mathcal P}}

\newcommand{\N}{\mathbb{N}}
\newcommand{\R}{\mathbb{R}}

\newcommand{\Z}{\mathbb{Z}}

\newcommand{\T}{\mathbb{T}}

\usepackage{mathrsfs}

\newcommand{\scA}{\mathscr{A}}
\newcommand{\scB}{\mathscr{B}}
\newcommand{\scC}{\mathscr{C}}

\newcommand{\scF}{\mathscr{F}}
\newcommand{\scI}{\mathscr{I}}

\newcommand{\scM}{\mathscr{M}}
\newcommand{\scO}{\mathscr{O}}

\newcommand{\scT}{\mathscr{T}}

\newcommand{\X}{\mathsf{X}}

\newcommand{\cring}{\cin\mathsf{Ring}}
\newcommand{\pcring}{\mathsf{P}\cin\mathsf{Ring}}
\newcommand{\PLCRS}{\mathsf{PL\cin RS}}

\newcommand{\LCRS}{\mathsf{L\cin RS}}

\newcommand{\Open}{\mathsf{Open}}

\DeclareMathAlphabet{\mathpzc}{OT1}{pzc}{m}{it}

\newcommand{\inv}{^{-1}}

\newcommand{\op}[1]{{#1}^{\mbox{\sf{\tiny{op}}}}}
\newcommand{\cin}{C^\infty}

\newcommand{\uu}[1]{\underline{#1}}
\newcommand{\ti}[1]{\tilde{#1}}

\DeclareMathOperator{\supp}{supp}

\DeclareMathOperator{\Hom}{Hom}
\DeclareMathOperator{\id}{id}

\DeclareMathOperator{\Spec}{Spec}
\DeclareMathOperator{\ev}{\mathsf{ev}}

\numberwithin{equation}{section}

\theoremstyle{definition}
\newtheorem{thm}{Theorem}[section]
\newtheorem{lemma}[thm]{Lemma}
\newtheorem{theorem}[thm]{Theorem}
\newtheorem*{corollary*}{Corollary}
\newtheorem*{claim*}{Claim}

\newtheorem{definition}[thm]{Definition}
\newtheorem{remark}[thm]{Remark}
\newtheorem{example}[thm]{Example}
\newtheorem{notation}[thm]{Notation}
\newtheorem {construction}[thm]{Construction}

\begin{document}
\title{ Poisson $\cin$-schemes.} 
 
\author{ Eugene Lerman}

% \thanks{typeset 
%   \DTMnow, file poisson-arxiv2}

\setcounter{tocdepth}{1}

\begin{abstract} We introduce Poisson $\cin$-rings and  Poisson local
  $\cin$-ringed spaces.  We show that the spectrum of a Poisson
  $\cin$-ring is an affine Poisson $\cin$-scheme.    We then discuss
  applications that include singular symplectic and Poisson
  reductions, singular quasi-Poisson reduction, coisotropic reduction
  and Poisson-Dirac subschemes.

  \end{abstract}
\maketitle
\tableofcontents

\section{Introduction}

In early 1970s Meyer \cite{Meyer} and, independently, Marsden and
Weinstein \cite{MW} developed a mathematical framework that connected
symmetries of Hamiltonian
systems with 
reduction of the number of degrees of freedom.  Their work lead to a
flurry of activity that persisted for several decades. It was quickly
observed that in many interesting cases the Meyer-Marsden-Weinstein
reduced spaces (a.k.a.\ symplectic quotients) are singular. Unlike
algebraic geometry where singular varieties and schemes make perfectly
good sense, a singular manifold is an oxymoron.  While singular spaces
arose in differential geometry long before symplectic reduction and a
wide variety of ways of dealing with singularities has been developed
(which I will not attempt to survey), there is no consensus about the
``best'' approach.

In the case of singular symplectic quotients the approaches roughly
fall into three categories:  ``geometric",
``algebraic'' (cf.\ \cite{AGJ}) and 
``homological'' \cite{Stasheff}.
In this paper we connect geometric and algebraic approaches using 
tools of algebraic geometry over $\cin$-rings \cite{Dubuc, Joy}.   
(The connections with homological approach should come out of derived differential geometry once it is well developed.)
From the $\cin$-ring point of view the algebraic approach results in 
Poisson $\cin$-rings (a concept that we introduce in this paper) while
geometric approach results in spaces with a structure sheaf of
Poisson $\cin$-rings.   We show that Dubuc's spectrum functor from
the category of  $\cin$-rings to the category of  local $\cin$-ringed spaces send Poisson
$\cin$-rings to Poisson ringed spaces (see Theorem~\ref{thm:3.19}).
This unifies algebraic and geometric views of reduction in symplectic
and Poisson geometry.   It also makes singular reduction a lot simpler
since there is no need to worry about stratifications.   With almost
no extra work we also make sense of singular  quasi-Poisson
reduction as Poisson $\cin$-schemes.  Other applications
include coisotropic reduction in Poisson geometry and Poisson-Dirac
submanifolds and,  more generally, subschemes.

\mbox{}\\
Recall that a $\cin$-ring is a product preserving functor from the
category of coordinate vector spaces $\R^n$, $n\in \N$, and smooth maps to the
category of sets \cite{MR, Joy}.  Equivalently a $\cin$-ring is a set
$\scA$ together with a collection of $n$-ary operations, $n\geq 0$,
one $n$-ary operation
\[
  f_\scA:\scA^n\to \scA
\]
for each smooth function $f\in \cin(\R^n)$.  These operations are
associative.  Maps between $\cin$-rings are natural transformation
between functors.  Equivalently they are maps of sets preserving the
operations.  It follows quickly from the definition that any
$\cin$-ring has an underlying $\R$-algebra.  A $\cin$-ring derivation
$v:\scA\to \scA$ of a $\cin$-ring $\scA$ is an $\R$-linear map that
satisfies a generalization of the Leibniz rule. Namely, for any $n$,
for any $f\in \cin(\R^n)$ and for any $a_1,\ldots, a_n\in \scA $ one
requires that
\begin{equation}
v(f_\scA(a_1,\ldots, a_n)) = \sum_{i=1}^n (\partial_i
f)_\scA(a_1,\ldots, a_n) \cdot v(a_i).
\end{equation}
For example, if $M$ is a manifold then its algebra $\cin(M)$ of smooth
functions is a $\cin$-ring with the operations given by
compositions. Namely,  for all $n$, all $f\in \cin(\R^n)$ and all $a_1,\ldots,
a_n\in \cin(M) = \cin(M, \R)$
\[
f_{\cin(M)} (a_1, \ldots, a_n) := f\circ (a_1,\ldots, a_n).
\]  
And then a vector field $v$ on a manifold $M$ is a $\cin$-ring
derivation since
\[
v( f\circ (a_1,\ldots, a_n) ) = \sum_{i=1}^n \big(\partial_i
f\circ (a_1,\ldots, a_n) \big)\cdot v(a_i).
\]
Recall that a Poisson algebra $A$ over a commutative ring $k$ is an
associative (and often commutative) $k$-algebra with a compatible Lie
algebra bracket $\{\cdot, \cdot\}:A\times A\to A$ such that for each
$a\in A$ the map $\{a, \cdot \} : A\to A$ is a $k$-algebra derivation.
Its is well-known \cite{Polishchuk, Ka1} that the algebro-geometric
spectrum of a Poisson algebra is a Poisson scheme, that is, a local
ringed space whose structure sheaf is a sheaf of Poisson algebras.
Poisson schemes have proven to be useful in algebraic geometry and
representation theory. However, if one starts with a Poisson manifold
$M$ and regards its algebra of functions $\cin(M)$ as Poisson algebra
over $\R$ then the associated Poisson scheme is not the original
manifold with its sheaf of smooth functions.  Of course one may feel
that there is no need to view Poisson manifolds as affine Poisson
schemes, and it's hard to argue with this point of view.  On the other
hand the Sniatycki-Weinstein algebraic reduction \cite{SW} produces a
Poisson algebra over $\R$, which is a Poisson algebra of a manifold
under a regularity assumption.  However, in general this Poisson
algebra need not be an algebra of functions on any space \cite{AGJ}.
It is natural to wonder if this algebra is the algebra of global
sections of a sheaf on some space.  This question was the original
motivation for undertaking the work in this paper.

Additionally, as was mentioned above, in symplectic and Poisson geometry there
often arise geometric objects that are singular but are not algebraic
varieties or schemes in any natural way.  Examples include singular symplectic quotients
\cite{ACG, SL}, singular Poisson quotients, closures of coadjoint
orbits of real Lie groups and, more generally, closures of symplectic
leaves in Poisson manifolds, and so on.   The techniques developed in
this paper give a novel point of view of such objects --- see Section~\ref{sec:ex}.

\subsection*{Organization of the paper} \quad In Section~\ref{sec:pcr}
we define Poisson $\cin$-rings.  We show that if a $\cin$-ring $\scA$
is Poisson and $I\subset \scA$ is a Poisson  ideal then the quotient
$\scA/I$ is again a Poisson $\cin$-ring.  The point here that the
induced bracket on $\scA/I$ is a $\cin$-ring biderivation, and not
just an $\R$-algebra biderivation.  We then give examples of such
quotients and relate one of them to the ``universal'' reduction of
Arms, Cushman and Gotay \cite{ACG}.  We also show that if $\scA$ is a
Poisson $\cin$-ring and $I\subset \scA$ is a multiplicative ideal then
its Poisson normalizer $N(I)$ is a $\cin$-subring of $\scA$.

In Section~\ref{sec:scheme} we review Joyce's construction of Dubuc's
spectrum functor
\[
\Spec: \op{\cring} \to \LCRS
\]  
($\cring$ is the category of $\cin$-rings, $\LCRS$ is the category of
local $\cin$-ringed spaces) and prove that if $\scA$ is a Poisson
$\cin$-ring then the affine $\cin$-scheme $\Spec(\scA)$ is Poisson
(i.e., its structure sheaf is a sheaf of {\em Poisson} $\cin$-rings).
We also prove that if $\varphi:\scA\to \scB$ is a map of Poisson
$\cin$-rings then $\Spec(\varphi):\Spec(\scB) \to \Spec(\scA)$ is
Poisson.

In Section~\ref{sec:ex} we discuss various applications of
Theorem~\ref{thm:3.19}: symplectic reduction, Poisson reduction,
quasi-Poisson reduction, coisotropic reduction in the setting of
Poisson manifolds, Snyaticki-Weinstein reduction as a Poisson
$\cin$-scheme and Poisson-Dirac submanifolds and subschemes.

In Section~\ref{sec:compute} we address the question of getting one's
hands on $\Spec(\scA)$.  We show that if a $\cin$-ring $\scA$ is finitely generated
(as a $\cin$-ring!) then the space $X_\scA$ of its points embeds as a
closed subspace in some $\R^N$ (the embedding depends on the choice of
generators).  If $G\times M\to M$ is an action of a Lie group on
a manifold $M$ then the $\cin$-ring $\cin(M)^G$ of invariant functions
may or may not be finitely generated (if $G$ is compact and the number
of orbit types is finite then $\cin(M)^G$ is finitely generated by a
nontrivial theorem of Schwarz \cite{Sch}).   Regardless of $\cin(M)^G$
being finitely generated we show that for proper actions (of not
necessarily compact) Lie groups the space $X_{\cin(M)^G}$
of points the $\cin$-ring $\cin(M)^G$ of invariants is homeomorphic to the
orbit space $M/G$ and consequently $\Spec(\cin(M)^G)$ is isomorphic to
$(M/G, \cin_{M/G})$ where $\cin_{M/G} \subset C^0_{M/G}$ is a sheaf of
``smooth'' functions on the orbit space. This implies, in particular, if $M$ is Poisson,
the action of $G$ on $M$ is Hamiltonian, $\mu:M\to \fg^\vee$ is an
equivariant moment map, $\alpha\in \fg^\vee$ and $I_\alpha \subset
\cin(M)^G$ is the vanishing ideal of the level set $\mu\inv (\alpha)$
then the set of points $X_{\cin(M)^G/I_\alpha}$ embeds into the orbit
space $M/G$ as a closed subspace and the $\cin$-Poisson scheme
$\Spec(\cin(M)^G/I_\alpha)$ ``is'' a closed subscheme of $(M/G, \cin_{M/G})$.  The same result, {\em mutatis
  mutandis} holds if $M$ is a quasi-Poisson manifold.

In Section~\ref{sec:CW} we discuss the Cushman-Weinstein conjecture
\cite{Eg-th, Eg-era, SL}
on embedding of symplectic quotients into Euclidean spaces as
``Poisson subvarieties.''

Finally note that the paper occasionally uses forward references.
For example Remark~\ref{rmrk:Milnors} refers to Lemma~\ref{lem:proper}.  These forward references  do not create any logical loops.

\section{Poisson $\cin$-rings} \label{sec:pcr}
In the introduction we sketched a definition of a $\cin$-ring.  We
omitted the definition of $\cin$-ring morphisms and of subrings of
$\cin$-rings.  See \cite{MR, Joy} for a more careful discussion.  In
place of such a discussion we introduce some notation, which is not
entirely standard.  It is, however, convenient.
\begin{notation} $\cin$-rings form a category $\cring$.  The objects of this
category are $\cin$-rings and morphisms are functions between
$\cin$-rings that preserve all the operations. That is, $\varphi:\scA \to
\scB$ is a map of $\cin$-rings if for every $n\geq 0$, all $f\in
\cin(\R^n)$ and all $a_1,\ldots, a_n\in \scA$
\[
\varphi(f_\scA (a_1,\ldots, a_n) = f_\scB (\varphi(a_1),\ldots, \varphi(a_n)).
\]
Here, as in the introduction, $f_\scA:\scA^n\to \scA$ is the operation
defined by the function $f$.
\end{notation}

Similarly a {\sf $\cin$-ringed space} is a pair $(X, \scO_X)$ where
$X$ is a topological space and $\scO_X$ is a sheaf of $\cin$-rings.
It is a {\sf local} $\cin$-ringed space if all the stalks of the sheaf
$\scO_X$ are local $\cin$-rings (see Definition~\ref{def:local_rs}
below).  A map/morphism $\uu{f}: (X, \scO_X) \to a (Y,\scO_Y)$ from a
local $\cin$-ringed space $(X, \scO_X)$ to a local $\cin$-ringed space
$ (Y,\scO_Y)$ is a pair of maps $(f, f_\#)$ where $f:X\to Y$ is a
continuous map and $f_\#: \scO_Y\to f_* \scO_\X$ is a map of sheaves
of $\cin$-rings.  Note that unlike ordinary algebraic geometry the map
on stalks induced by $f_\# $ automatically preserves the unique maximal ideals.
\begin{notation}
We denote the category of local $\cin$-ringed spaces and their maps by $\LCRS$.
\end{notation}  
\begin{definition}
A {\sf module } over a $\cin$-ring $\scA$ is a module over the
underlying $\R$-algebra (cf.\ \cite{Joy, LdR}).
\end{definition}

\begin{definition}
Let $\scM$ be a module over a $\cin$-ring $\scA$.  A {\sf $\cin$-ring
derivation} (or a {\sf $\cin$-derivation} for short) with values in
the module 
$\scM$ is an $\R$-linear map $v:\scA\to \scM$ so that for any $n$, any $f\in
\cin(\R^n)$ and any $a_1,\ldots, a_n\in \scA$
\[
v\left(f_\scA(a_1,\ldots,a_n)\right) = \sum_{i=1}^n (\partial_i
  f)_\scA(a_1,\ldots, a_n)\cdot v (a_i).
\]  
\end{definition}
\begin{definition} \label{def:Poisson_ring}
A {\sf Poisson $\cin$-ring} is a $\cin$-ring $\scA$ together with a
Poisson bracket $\{\cdot, \cdot\}: \scA\times \scA \to \scA$ on the
underlying $\R$-algebra so that for any $a\in \scA$ the map
\[
ad(a)\equiv \{a, \cdot \}: \scA\to \scA, \quad b\mapsto \{a, b\}
\]  
is a $\cin$-ring derivation: for any $n>0$, any $f\in \cin(\R^n)$ and any $a_1,\ldots, a_n\in
\scA$
\begin{equation} %
  \left\{a, f_\scA(a_1,\ldots,a_n) \right\} = \sum_{i=1}^n (\partial_i
  f)_\scA(a_1,\ldots, a_n)\cdot \{ a, a_i\}.
 \end{equation} 
\end{definition}

\begin{example}
Let $(P, \pi\in \Gamma \Lambda^2 TM)$ be a Poisson manifold and
$\{f,g\} = \langle df\wedge dg , \pi\rangle $ the corresponding
bracket.  Then $(\cin(P), \{ \cdot, \cdot \})$ is a $\cin$-Poisson
ring.  This is because $ad(f)$ is a vector field for every $f\in \cin(P)$.
\end{example}  

\begin{definition}
A {\sf map} of Poisson $\cin$-rings is a map of $\cin$-rings that
preserve the Poisson brackets.
\end{definition}

We define subalgebras of Poisson $\cin$-rings as in the
case of ordinary Poisson algebras  {\em
  mutatis mutandis}:
\begin{definition}
A {\sf Poisson $\cin$-subring} (i.e., a Poisson subalgebra) of a Poisson $\cin$-ring $(\scA, \{\cdot,
\cdot\})$ is a $\cin$-subring $\scB$  of $\scA$ which is closed under
the Poisson brackets.
\end{definition}

\begin{example}
Let $(M, \pi)$ be a Poisson manifold with an action of a Lie group $G$
that preserves the Poisson tensor $\pi$.  Then the algebra $\cin(M)^G$ of
$G$-invariant functions is a $\cin$-subring of $\cin(M)$ and a Poisson
$\cin$-subring of $(\cin(M), \{\cdot, \cdot\})$.
\end{example}

Recall that an {\sf ideal} $I$ in a $\cin$-ring $\scA$ is an ordinary
ideal in the $\R$-algebra underlying $\scA$.  Recall also that if $I$
is an ideal in the $\cin$-ring $\scA$ then the quotient $\scA/I$ is
naturally a $\cin$-ring with the operations
\[
f_{\scA/I}:(\scA/I)^n \to \scA/I
\]
given by
\[
f_{\scA/I} (a_1 +I,\ldots, a_n+I)= f_\scA(a_1,\ldots, a_n) +I
\]
(for all $n\geq 0$, all $f\in \cin(\R^n)$ and all $a_1+I,\ldots,
a_n+I\in \scA/I$). The fact that the operations are well-defined is a consequence of Hadamard's lemma; see \cite[p.~17]{Jet}.
\begin{definition}
A {\sf Poisson ideal} of a Poisson $\cin$-ring $\scA$ is an
ideal in (the $\R$-algebra  underlying) $\scA$ so that
\[
\{f, a\} \in I \qquad \textrm{ for all } f\in I, a\in \scA.
\]  
\end{definition}
\begin{definition} \label{def:poisson_sheaf}
A sheaf $\scO$ of $\cin$-rings on a space $X$ is a {\sf sheaf of
  Poisson $\cin$-rings} if for every open set $U\subset X$ the
$\cin$-ring $\scO(U)$ is Poisson and the restriction maps are maps of
Poisson $\cin$-rings.
\end{definition}

\begin{lemma} \label{lem:3.4}
Let $I$ be a Poisson ideal in a Poisson $\cin$-ring $\scA$.  Then the
quotient $\cin$-ring $\scA/I$ is Poisson 
and the projection map $p: \scA \to \scA/I$ is map of
Poisson $\cin$-rings.
\end{lemma}  
\begin{proof} This 
is an easy consequence of Lemma~\ref{lem:2.9}
below.
\end{proof}
\begin{lemma} \label{lem:2.9}
Let $v:\scA\to \scA$ be a $\cin$-ring derivation of a $\cin$-ring
$\scA$ and $I\subseteq \scA$ an ideal with $v(I)\subset I$. Then the
induced map
\[
\bar{v}:\scA/I\to \scA/I, \qquad \bar{v}(a+I): = v(a)+ I
\]  
is a $\cin$-ring derivation.
\end{lemma}  

\begin{proof}
Given $n$, $f \in \cin(\R^n) $ and $a_1,\ldots, a_n \in \scA$ we
compute:
\[
\begin{split}
\bar{v} (f_{\scA/I} (a_1 +I,\ldots, a_n+I) =&\bar{v}( f_\scA(a_1,\ldots,a_n) +I)\\
=&v(f_\scA(a_1,\ldots,a_n)) +I\\
=& \left( \sum (\partial_i f)_\scA(a_1,\ldots, a_n) v(a_i) \right)+I\\
  =&  \sum( (\partial_i f)_\scA(a_1,\ldots, a_n) +I)( v(a_i) +I)\\
=&  \sum (\partial_i f)_{\scA/I} (a_1+I,\ldots, a_n+I) \bar{v}(a_i+I).
\end{split}
\]  
\end{proof}

The next lemma is useful for producing natural examples of Poisson
ideals.  It is a generalization of 
\cite[Lemma 2]{BL}, which, in turn, is a variation on the ideas in \cite{ACG}.

\begin{lemma} \label{lem:2.10}
Let $M$ be  a Poisson manifold,
and ${\cA}$ be a Poisson $\cin$-subring of $C^\infty (M)$.  Suppose that the
Hamiltonian flows of functions in ${\cA}$ preserve a  subset $X$ of the
manifold $M$.  Then the vanishing  ideal
\[
	{\cI}(X) := \{f \in {\cA} : f|_X = 0\}
\]
of  $X$ is a Poisson ideal of ${\cA}$.
\end{lemma}

\begin{proof}
Let  $f$ be in ${\cA}$, $x$ be a point in $X$ and $h$ be in the ideal ${\cI}(X)$.
Let $\gamma (t)$ denote the integral curve of the Hamiltonian vector field
of $f$ with $\gamma (0) =x$.
Then $\gamma (t)$ is in $X$ and consequently  $h(\gamma (t))= 0$ for all $t$.
Differentiation with respect to $t$ yields
$$
0 = \left. \frac{d}{dt}  \right|_0 h(\gamma (t)) = \{f, h\} (x)
$$
Thus $ \{f, h\}|_X =0$, i.e., $ \{f, h\}$ is in the ideal ${\cI}(X)$.
\end{proof}

\begin{example} Let $\fg$ be a real Lie algebra.  Then its dual
  $\fg^\vee$ is canonically a Poisson manifold.  Let $\scO\subset \fg^\vee$
  be a coadjoint orbit.  Since the flows of Hamiltonian vector fields
  on $\fg^\vee$ preserve the orbit $\scO$, the ideal of functions $I =\{ f \in \cin
  (\fg^\vee) \mid f|_\scO = 0\}$ is a Poisson ideal by Lemma~\ref{lem:2.10}.
By Lemma~\ref{lem:3.4} the $\cin$-ring $\cin(\fg^\vee)/I$ is a Poisson
$\cin$-ring.  It's the ring of (Whitney) smooth functions on the
closure $\overline{\scO}$ of the orbit $\scO$: $\cin(\fg^\vee)/I =
\cin(\overline{\scO}$).    It follows that $\cin(\overline{\scO})$ is a Poisson $\cin$-ring.
\end{example}
We can easily generalize the example above.
\begin{example} \label{ex:2.12}
Let $(P, \pi)$ be a Poisson manifold and $X\subset P$ a closed subset which is
a union of symplectic leaves of the Poisson tensor $\pi$.  Then the
ideal $I_X \subset \cin(P)$ of functions that vanish on $X$ is a
Poisson ideal in the $\cin$-Poisson algebra $\cin(P)$.
\end{example}  

Another application of Lemma~\ref{lem:2.10} gives  us 
the so called ``universal reduction''  of Arms, Cushman and Gotay \cite{ACG}:

\begin{example} \label{ex:ACG}
Let $(M, \omega)$ be a symplectic manifold with a Hamiltonian action of a Lie group $G$ and let $\mu:M\to \fg^\vee$ denote the corresponding equivariant moment map.
Let $\alpha \in\fg^\vee$ be a point.  We make no assumption on the stabilizer $G_\alpha$ of $\alpha$.  Consider the set
\[
I_\alpha := \{f \in \cin(M)^G\mid f|_{\mu\inv (\alpha)} =0\},
\]
the set of all invariant functions that vanish on the level set $\mu\inv (\alpha)$.   Since the moment map $\mu$ is constant along the 
integral curves of all invariant functions $f\in \cin(M)^G$, the flows of invariant functions preserve $\mu\inv (\alpha)$.  By Lemma~\ref{lem:2.10} the ideal $I_\alpha$ is a Poisson ideal in the Poisson $\cin$-ring $\cin(M)^G$.   Therefore the quotient $\cin$-ring $\cin(M)^G/I_\alpha$ is a Poisson $\cin$-ring.
\end{example}

\begin{remark} \label{rmrk:2.14}
Any element of the quotient $\cin$-ring $\cin(M)^G/I_\alpha$  gives rise to a well-defined continuous function on the quotient 
space $\mu\inv (\alpha)/G_\alpha$ where $G_\alpha$ is the stabilizer
of $\alpha\in \fg^\vee$ (under the coadjoint action).  In the next
section (Theorem~\ref{thm:3.19}) we will see that one can associate to
any Poisson $\cin$-ring $\scA$ a space $X_\scA$ with a sheaf of
Poisson $\cin$-rings $\scO_\scA$.  If the action of a Lie group $G$ on
a manifold $M$ is nice enough, the two spaces turn out to be homeomorphic.  We will prove it in Section~\ref{sec:compute}.
\end{remark}

\begin{definition}
Let $(\scA, \{\cdot, \cdot\})$ be a Poisson $\cin$-ring.  The {\sf
  normalizer} $N(S)$ of a subset $S\subset \scA$ is the set
\[
  N(S) := \{ a\in \scA \mid \{s, a\}\in S\textrm{ for all }s\in S\}.
\]
\end{definition}

\begin{lemma} \label{lem:SW}
Let $(\scA, \{\cdot, \cdot\})$ be a Poisson $\cin$-ring, $\scI \subset
\scA$ an ideal in the commutative $\R$-algebra underlying $\scA$. 
The normalizer  $N(\scI)$ of $\scI$ is a
$\cin$-subring of $\scA$. Consequently
\[
\scB: = N(\scI)/(\scI \cap N(\scI))
\]
is a Poisson $\cin$-ring.
\end{lemma}  

\begin{proof}
Given $n>0$, $h\in \cin(\R^n)$, $f_1, \ldots f_n \in N(\scI)$ and
$j\in \scI$ we want to show that $\{j, h_\scA(f_1,\ldots f_n)\} \in
N(\scI)$.  
 Since 
$\{j, \cdot\} :\scA\to \scA$ is a $\cin$-ring derivation,
\[
\{j,  h_\scA(f_1,\ldots f_n)\} = \sum_{i=1}^n (\partial_i h)_\scA
(f_1,\ldots f_n) \{j, f_i\}.
\]  
Since $\scI$ is an ideal and since $\{j, f_i\}\in \scI$ for all $i$, it follows that $\{j, h_\scA(f_1,\ldots f_n)\} \in
N(\scI)$ as desired.
\end{proof} 

\begin{example}[Sniatycki-Weinstein algebraic reduction \cite{SW}] \label{ex:SW}
Let $(M, \omega)$ be a symplectic manifold with a Hamiltonian action
of a connected Lie group $G$ and associated equivariant moment map
$\mu:M\to \fg^\vee$.  Denote the canonical pairing $\fg^\vee\times \fg\to
\R$ by $\langle \cdot, \cdot \rangle$.  For every element $X$ in the
Lie algebra $\fg$ we have a smooth function $\mu_X := \langle \mu
,X\rangle:M\to \R$.  Let $\scI $ be the ideal in the $\cin$-ring
$\cin(M)$ generated by the set $\{ \mu_X \mid X\in \fg\}$.  Since
$\mu$ is equivariant, $\{\mu_X, \mu_Y\} = \mu_{[X,Y]}$ for all $X,
Y\in \fg$.  Consequently $\scI$ is a Poisson subalgebra of $\cin(M)$.
Sniatycki and Weinstein define the {\sf reduced algebra} associated to
the data $(M, \omega, \mu)$ to be the algebra
\[
    \scB = (\cin(M)/ \scI)^G
\]
of $G$-invariant elements of the quotient $\cin(M)/\scI$.    They then
argue \cite[Lemma 2]{SW} that $\scB \simeq N(\scI)/\scI$, where as in
Lemma~\ref{lem:SW}, $N(\scI)$ is the normalizer of the ideal $\scI$ in
the Poisson algebra $\cin(M)$.   By Lemma~\ref{lem:SW}, the Poisson
algebra $\scB$ is a Poisson $\cin$-ring.
\end{example}

Sniatycki and Weinstein ({\em op. cit.}) give an example of their reduced Poisson
algebra that is not the Poisson algebra of a manifold. Arms, Gotay and
Jennings \cite{AGJ}  give further examples where the reduced Poisson
algebra is not an algebra of functions on a space.

As was noted in Remark~\ref{rmrk:2.14} one can attach to the
 Sniatycki-Weinstein reduced algebra $\scB = N(\scI)/\scI$ (which is a Poisson $\cin$-ring) 
 a space with a sheaf of Poisson $\cin$-rings.  See
 Example~\ref{ex:SW2} for further discussion.

\section{Poisson $\cin$-schemes} \label{sec:scheme}

The main result of this section is Theorem~\ref{thm:3.19} and its relative version, Lemma~\ref{lem:3.23}: the
spectrum of a Poisson $\cin$-ring is a Poisson $\cin$-scheme and if $\varphi:\scA \to \scB$ is a map of Poisson $\cin$-rings then $\Spec(\varphi) :\Spec(\scB) \to \Spec(\scA)$ is a map of Poisson $\cin$-schemes.  
 The proofs use the details of Joyce's construction \cite{Joy}
of Dubuc's spectrum functor $\Spec$ from the category $\op{\cring}$
  (the category opposite to the category of $\cin$-rings) to local
$\cin$-ringed spaces $\LCRS$ (which will be defined below).  
Note that unlike the
corresponding functor in algebraic geometry, Dubuc's 
\[
\Spec:
\op{\cring} \to \LCRS
\]
 is not
fully faithful.
We start by reviewing some definitions that are used in Joyce's construction.
\begin{definition} \label{def:local_rs}
An {\sf
  $\R$-point} of a $\cin$-ring $\scA$ is a map of $\cin$-rings
$p:\scA \to \R$.

A $\cin$-ring $\scA$ is {\sf local} if it has a
unique $\R$-point.  Equivalently there exists a unique maximal ideal
$M\subseteq \scA$ so that the quotient $\scA/M$ is isomorphic to $\R$.

A $\cin$-ringed space $(X, \scO_X)$ is a {\sf local $\cin$-ringed
  space} if all the stalks of the structure sheaf $\scO_X$ are local
$\cin$-rings.
\end{definition}

By analogy with algebraic geometry we define a Poisson
local $\cin$-ringed space as follows.

\begin{definition}
A {\sf Poisson local $\cin$-ringed space} % 
is a local $\cin$-ringed
space $(X, \scO_X)$ so that the structure sheaf $\scO_X$ is a sheaf of
Poisson $\cin$-rings (cf.\ Definition~\ref{def:poisson_sheaf}).
\end{definition}

\begin{notation} \label{not:2.13}
We denote the set of all $\R$-points of a $\cin$-ring $\scA$ by
$X_\scA$.
Thus
\[
X_\scA: =\{ p:\scA\to \R\mid p \textrm{ is a map of $\cin$-rings}\}
\equiv \Hom(\scA, \R).
\]  
\end{notation}

\begin{remark} \label{rmrk:Milnors}
It is well-known that if $M$ is a (second-countable Hausdorff)
manifold then the map
\begin{equation} \label{eq:2.1}
M\to \Hom_{\R\text{-alg}}(\cin(M), \R),\qquad p\mapsto \ev_p
\end{equation}
is a bijection. Here and below $\ev_p (f) = f(p)$ for all $f\in
\cin(M)$, and $\Hom_{\R\text{-alg}}(\cin(M), \R)$ is the set of $\R$-algebra homomorphisms . The bijectivity of \eqref{eq:2.1} is a theorem of Pursell \cite[Chapter
8]{Pursell}. It is often referred to as Milnor's exercise.  

Note that the surjectivity of \eqref{eq:2.1} follows from the fact that
$\cin(M)$ contains a proper function  combined with Lemma~\ref{lem:proper}.  The injectivity of \eqref{eq:2.1} is the fact that smooth
functions on Hausdorff manifolds separate points.

It is easy to check that for any point $p\in M$ the map $\ev_p: \cin(M)\to \R$ is actually a map of $\cin$-rings, and not just a map of $\R$-algebras.  Consequently 
\[
\Hom_{\R\text{-alg}}(\cin(M), \R) = \Hom_{\cring}(\cin(M), \R).
\]
\end{remark}

\begin{construction}[The Zariski topology $\scT_\scA$ on the set
  $X_\scA$ of $\R$-points of a $\cin$-ring $\scA$] \label{constr:Z}
  The set $X_\scA$ of $\R$-points of a $\cin$-ring $\scA$ comes
  equipped with a natural topology, the {\sf Zariski topology}.  It is
  defined as follows: for $a\in \scA$ let
\[
  U_a: = \{p\in X_\scA \mid p(a) \not = 0\}.
\]
Since for a point $p$ of $\scA$ and $a,b\in \scA$
\[
p(ab) \not = 0 \quad \Leftrightarrow \quad p(a) \not = 0 \textrm{ and
} p(b) \not = 0,
\]
we have $U_{ab} = U_a \cap U_b$. Hence the collection  $
\{ U_a\}_{a\in \scA}$ of sets
is a basis for a topology on $X_\scA$, which we denote by $\scT_\scA$.
\end{construction}

\begin{remark} \label{rmrk:2.Whit}
By a theorem of Whitney any closed subset $C$ of a smooth (Hausdorff
and paracompact) manifold $M$
is the set of zeros of some smooth function $f\in \cin(M)$.
It follows that the bijection \eqref{eq:2.1} is a homeomorphism
between the manifold $M$ with its given topology and the set
$X_{\cin(M)}$ of $\R$-points of $\cin(M)$ with the Zariski topology $\scT_{\cin(M)}$.

Thus, in particular, one can recover the manifold $M$ as
a topological space simply from the $\R$-algebra $C^\infty(M)$.

\end{remark}  

\begin{remark} \label{rmrk:3.6}
Unlike the Zariski topology in algebraic geometry the topology $\scT_\scA$ of 
Construction~\ref{constr:Z} is much nicer.  It is Hausdorff and, additionally, it is 
 completely
regular: for any closed set $C\subset X_\scA$
and for any point $x\not \in C$ there is a continuous function
$f:X_\scA\to \R$ with $f(x) \not = 0$ and $f|_C = 0$, see \cite{Joy}.
\end{remark}
To construct the functor $\Spec$ on objects, given a $\cin$-ring
$\scA$  we need to equip the 
topological space $X_\scA$ of its points with a sheaf $\scO_\scA$ of
$\cin$-rings. To construct this structure sheaf  $\scO_\scA$ we need to recall a
few things about localizations of $\cin$-rings.

\begin{lemma} \label{lem:2.21}
Given a $\cin$-ring $\scA$ and a set $\Sigma$ of nonzero elements of $\scA$
there exists a $\cin$-ring $\scA\{\Sigma\inv\}$ and a map $\gamma:
\scA\to \scA\{\Sigma\inv\}$ of $\cin$-rings with the following universal property:
\begin{enumerate}
\item $\gamma(a)$ is invertible in $\scA\{\Sigma\inv\}$ for all $a\in
  \Sigma$ (i.e., there is $b\in \scA\{\Sigma\inv\}$ such that $\gamma(a) b =1$) and 
\item  for
any map $\varphi:\scA\to \scB$ of $\cin$-rings so that $\varphi(a)$ is
invertible in $\scB$ for all $a\in \Sigma$ there exists a unique map
$\overline{\varphi} :\scA\{\Sigma\inv\} \to \scB$ of $\cin$-rings making the diagram
\[
  \xy
(-10,10)*+{\scA }="1";
(14,10)*+{\scB}="2";
(-10, -6)*+{\scA\{\Sigma\inv\}}="3";
{\ar@{->} ^{\varphi} "1";"2"};
{\ar@{->}_{\gamma} "1";"3"};
{\ar@{->}_{\overline{\varphi}} "3";"2"};
\endxy
\]  
commute.
\end{enumerate}
\end{lemma}
\begin{proof} See  \cite[p.~23]{MR}.
\end{proof}
\begin{definition}
  We refer to the map 
$\gamma:\scA \to \scA\{\Sigma\inv\}$ of Lemma~\ref{lem:2.21} as  a {\sf localization} of
the $\cin$-ring $\scA$ at the set $\Sigma$.   
\end{definition}
\begin{remark} A localization of a $\cin$-ring $\scA$ at a set
  $\Sigma$ is unique up to a unique isomorphism, so we can speak about
  {\em the} localization of $\scA$ at $\Sigma$.
\end{remark}

\begin{notation}\label{not:2.22}
Let $\scA$ be a $\cin$-ring. By Lemma~\ref{lem:2.21}, for an
$\R$-point $x:\scA\to \R$ of a $\cin$-ring $\scA$ there exists a
localization of $\scA$ at the set
\[
\{x\not =0\}:= \{a\in \scA \mid x(a) \not = 0\}
\]
We set
\[
\scA_x:= \scA\{ \{x\not =0\}\inv\}.
\] 
and denote the corresponding localization map by
\begin{equation} \label{eq:pix}
  \pi_x:\scA\to \scA_x
\end{equation}
\end{notation}
\noindent
Joyce proves \cite[Proposition~2.14]{Joy} that $\pi_x:\scA \to \scA_x$ is surjective
with $I_x:= \ker \pi_x$ given by
\begin{equation} \label{eq:5.13}
I_x:= \{ a \in \scA \mid \textrm{ there is } d\in \scA \textrm{ so
  that } x(d)\not =0 \textrm{ and } ad= 0\}.
\end{equation}
We think of $\scA_x$ as the ring of germs of elements of $\scA$ at the
point $x$.

\begin{remark} \label{rmrk:3.12} In case of $\scA = \cin(\R^n)$ and
  $x:\cin(\R^n) \to \R$ (which is the evaluation at some point
  $p\in \R^n$ by Remark~\ref{rmrk:Milnors}) 
  the localization $(\cin(\R^n))_x$ is isomorphic to the usual ring of
  germs of functions at $p$.  This is because both rings are
  localizations of $\cin(\R^n)$ at the same set; see \cite[Example
  2.15]{Joy}.
  
  Alternatively, there is a surjective map from $\cin(\R^n) $ to the ring of germs of smooth functions at the point $p\in \R^n$. The kernel of this map consists of functions that vanish on a neighborhood of $p$, which is exactly the ideal $I_{\ev_p}$.
\end{remark}
\begin{remark}
  The localizations $\scA_x$ are local rings.  This is easy to see.
  Note first that for any $a\in \{x\not =0\}$, $x(a)$ is invertible in
  $\R$, hence $x:\scA\to \R$ gives rise to $\overline{x}:\scA_x\to \R$
  with $ x = \overline{x}\circ \pi_x$.  Moreover, for any $c\in \scA$,
  $\pi_x(c) \not \in \ker{\overline{x}}$ if and only if
  $x(c) \not = 0$ if and only if $\pi_x(c)$ is invertible in
  $\scA_x$. Hence $\scA_x \smallsetminus \ker{\overline{x}}$ consists
  of units of $\scA_x$ and therefore $\ker{\overline{x}}$ is a unique
  maximal ideal in $\scA_x$.
\end{remark}

\begin{construction} \label{construction:1}  Given a $\cin$-ring
$\scA$ we construct the corresponding affine $\cin$-scheme
$\Spec(\scA)$.  In Construction~\ref{constr:Z} we defined  a
topology on the set $X_\scA$ of $\R$-points of $\scA$.  We now
construct the structure sheaf $\scO_\scA$ on $X_\scA$.

We start by defining  a candidate etale space $S_\scA$ of the
sheaf $\scO_\scA$:
\[
S_\scA:= \bigsqcup _{x\in X_\scA} \scA_x \equiv \bigsqcup _{x\in X_\scA} \scA/I_x.
\]
The set $S_\scA$ comes with the evident surjective map $\varpi:
S_\scA\to X_\scA$ defined by $\varpi( s) = x$ for all $s\in
\scA_x\hookrightarrow S_\scA$. For any $a\in \scA$ we get a section $\frs_a: X_\scA
\to S_\scA$ of $\varpi$:
\begin{equation} \label{eq:2.6}
\frs_a(x) = \pi_x (a) \equiv a_x,
\end{equation}
where, as before, $\pi_x:\scA \to \scA_x$ is the localization map
\eqref{eq:pix}.  The collection of sets
\[
\{ \frs_a (U)\mid a\in \scA, U\in \Open(X_\scA) \}
\]  
forms a basis for a topology on $S_\scA$.  In this topology the
projection $\varpi: S_\scA\to X_\scA$ is a local homeomorphism and all
sections $\frs_a:X_\scA \to S_\scA$ are continuous.

{\em We
define the structure sheaf $\scO_\scA$ of $\Spec(\scA)$ to be the sheaf
of continuous sections of  $\varpi: S_\scA \to X_\scA$.}     Equivalently
\begin{equation} \label{eq:2.4}
\begin{split}  
\scO_\scA (U) =   &\{ s:U\to \bigsqcup _{x\in U} \scA_x\mid \textrm{ there
  is an open cover } \{U_\alpha\}_{\alpha \in A} \textrm{ of } U\\ 
  &
\textrm{ and } \{a_\alpha\}_{\alpha\in A} \subset \scA \textrm{ so that }
s|_{U_\alpha} = \frs_{a_\alpha}|_{U_\alpha} \textrm{ for all } \alpha \in A\}.
\end{split}
\end{equation}
for every open subset $U$ of $X_\scA$.  The $\cin$-ring structure on
the sets $\scO_\scA(U)$ is defined pointwise.

It follows from \eqref{eq:2.4} that the sheaf $\scO_\scA$ is a
sheafification of a presheaf $\cP_\scA$ defined by
\begin{equation}\label{eq:sec}
  \cP_\scA(U) := \{\frs_a|_U \mid a\in \scA\}.
\end{equation}

It turns out that the stalk of the sheaf $\scO_\scA$ at a point
$x\in X_\scA$ is (isomorphic to) $\scA_x$ (\!\cite[Lemma~4.18]{Joy}).  Consequently the pair
$(X_\scA, \scO_\scA)$ is a {\em local} $\cin$-ringed space. 
\end{construction}

\begin{remark} \label{rmrk:2.26}
Note that the canonical map $\eta: \cP_\scA\to \scO_\scA$ from a presheaf to its
sheafification is simply the inclusion: $\eta(\frs_a) = \frs_a$.

Note also that for any open subset $U$ of $X_\scA$ the map
\[
\scA \to \cP_\scA(U), \qquad a\mapsto \frs_a|_U
\]
is a surjective map of $\cin$-rings with the kernel $\bigcap_{x\in
  U}I_x$ (the ideals $I_x$ are given  by \eqref{eq:5.13}).  Thus
\[
\cP_\scA (U) \simeq \scA/\bigcap_{x\in U}I_x.
\]

Finally note that for any $\cin$-ring $\scA$ we get a map of
$\cin$-rings
\[
\scA \to \cP(X_\scA), \qquad a \mapsto \frs_a.
\]
which we may think of as a map
\[
\varepsilon_\scA: \scA \to \Gamma(\scO_\scA), \qquad \varepsilon_\scA(a) = \frs_a.
\]  
In general the map $\varepsilon_\scA$ is neither injective nor
surjective.  For example, $\ker \varepsilon_\scA = \bigcap _{x\in
  X_\scA}I_x$, which may be nonzero.
\end{remark}  
\begin{remark} \label{rmrk:2.27}
In the case where the $\cin$-ring $\scA$ in
Construction~\ref{construction:1} is the ring $\cin(M)$ of smooth
functions on a (second countable Hausdorff) manifold $M$, the space $X_\scA$ of points of $\scA$  is canonically
homeomorphic to the topological space $M$ underlying the manifold $M$.
 More generally, as a local
$\cin$-ringed space $\Spec(\cin(M))$ is isomorphic to $(M, \cin_M)$
where $\cin_M$ denotes the sheaf of smooth functions on the manifold
$M$.   Moreover there is a unique  isomorphism
$\uu{\varphi}_M: (M, \cin_M) \to \Spec(\cin(M))$ that makes the diagram
\[
\xy
(-10,10)*+{\cin(M)}="1";
(14,10)*+{\Gamma(\cin_M)}="2";
(14, -10)*+{\Gamma(\scO_{\cin(M)})}="3";
{\ar@{->} ^{\id} "1";"2"};
{\ar@{->}_{\varepsilon_{\cin(M)}} "1";"3"};
{\ar@{<-}_{\Gamma(\uu{\varphi}_M)} "3";"2"};
\endxy
\]
commute.   From now on we suppress the isomorphism
$\uu{\varphi}_M: (M, \cin_M) \to \Spec(\cin(M))$ for all manifolds
$M$.  That is, in effect, we set
\[
\Spec(\cin(M)) = (M, \cin_M)
\]  
for all manifolds $M$.
\end{remark}

\begin{lemma} \label{lem:3.18}
Let $\scA$ be a Poisson $\cin$-ring and $x:\scA\to \R$ an $\R$-point.
Then the ideal $I_x$ defined by \eqref{eq:5.13} is a Poisson ideal.
\end{lemma}

\begin{proof}
We need to show that for any $a\in I_x$ and any $b\in \scA$ the
bracket $\{a,b\}$ is again in $I_x$.

Since $a\in I_x$ there is
an element $d\in
\scA$ so that $x(d)\not = 0$ and $ad = 0$ (cf.\ Remark~\ref{rmrk:3.12}.  And then
\[
0=\{b, ad\} = \{b,a\}d + a \{b,d\}.
\]
Since $a\in I_x$ and $I_x$ is an ideal, $a\{b,d\}\in
I_x$. Consequently $\{b,a\} d\in I_x$.  Therefore there is $d'\in
\scA$ with $x(d')\not = 0$ and $\{b,a\}dd' = 0$.  Since $0 \not =x(d)
x(d') = x(dd')$, it follows that $\{b,a\}\in I_x$ and we are done.
\end{proof}

\begin{theorem} \label{thm:3.19}
Let $\scA$ be a Poisson $\cin$-ring
(Definition~\ref{def:Poisson_ring}).  Then the structure sheaf
$\scO_\scA$ of the affine $\cin$-scheme  $\Spec(\scA) =
(X_\scA, \scO_\scA)$ is a sheaf of Poisson $\cin$-rings.  That is, $\Spec(\scA)$ is a Poisson $\cin$-scheme.
\end{theorem}

\begin{proof}
Since the structure sheaf $\scO_\scA$ is the sheafification of the
presheaf $\cP_\scA$, it is enough to show that for every open set
$U\subset X_\scA$ the $\cin$-ring $\cP_\scA(U)$ is Poisson and that
for pairs of open sets $V\subseteq U$
the restriction maps  $\cP_\scA(U) \to \cP_\scA(V)$ are Poisson.
Recall from Remark~\ref{rmrk:2.26} that  for every open $U\subset
X_\scA$ the map
\[
\scA \to \cP(U), \qquad a\mapsto \frs_a|_U
\]
is surjective with the kernel $\bigcap_{x\in U}I_x$.   By
Lemma~\ref{lem:3.18} the  ideals
$I_x$ are  Poisson for all $\R$-points $x$.  Hence the intersection
$\bigcap_{x\in U}I_x$ is a Poisson ideal, and therefore the
$\cin$-ring $\cP(U)$ is Poisson.   Note that the bracket on $\cP(U)$
is given by
\[
\{ \frs_a|_U, \frs_b|_U\} = \frs_{\{a,b\}}|_U.
\]  
It follows that the restriction maps
\[
\cP_\scA(U)\to \cP_\scA(V),\qquad \frs_a|_U \mapsto (\frs_a|_U)|_V  = \frs_a|_V
\]
are Poisson.
\end{proof}

\begin{notation}
Poisson $\cin$-rings form a category, which we denote by $\pcring$. The
objects of $\pcring$ are Poisson
$\cin$-rings, the  morphisms are maps of $\cin$-rings that preserve
brackets.
\end{notation}
\begin{notation}
Poisson local $\cin$-ringed spaces also from a category:
the objects are Poisson local $\cin$-ringed spaces. We denote it by
$\PLCRS$.  A map in this category  from a Poisson local $\cin$-ringed
space $(X, \scO_X)$ to  another Poisson local $\cin$-ringed space $(Y, \scO_Y)$
is a map $\uu{f} = (f, f_\#): (X, \scO_X) \to (Y, \scO_Y)$ of local
$\cin$-ringed spaces so that $f_\#: \scO_Y\to f_* \scO_X$ is a map of
sheaves of 
Poisson $\cin$-rings.  That is, for every open set $U\subset X$ the
map
\[
f_{\#, U} : \scO_Y(U)\to \scO_X(f\inv (U))
\]
of $\cin$-rings preserves the Poisson bracket.
\end{notation}

We can restrict Dubuc's spectrum functor $\Spec:\op{\cring} \to
\LCRS$ to the subcategory $\pcring$ of Poisson $\cin$-rings.  We
expect the image  of $\Spec|_{\pcring}$ to land in the category $\PLCRS$ of Poisson local $\cin$-ringed spaces.
To prove this we need to recall Joyce's construction of
$\Spec$ on morphisms.  As a first step we recall  a lemma,
which is well-known to experts.

\begin{lemma} \label{lem:2.25}
Let $\varphi: \scA \to \scB$ be a map of $\cin$-rings and $x:\scB\to
\R$ a point.  Then $\varphi$ induces a unique map $\varphi_x:
\scA_{x\circ \varphi} \to \scB_x$ of $\cin$-rings on their respective  localizations making the diagram
\begin{equation} \label{eq:2.4n}
  \xy
(-10,6)*+{\scA }="1";
(14,6)*+{\scB}="2";
(-10, -6)*+{\scA_{x\circ \varphi}}="3";
 (14, -6)*+{\scB_x}="4";
{\ar@{->}_{\pi_{x\circ \varphi}} "1";"3"};
{\ar@{->}^{\pi_x} "2";"4"};
{\ar@{->}^{\varphi} "1";"2"};
{\ar@{->}_{\varphi_x} "3";"4"};
\endxy
\end{equation}
commute.
\end{lemma}

\begin{proof}
Since $\pi_y: \scA \to \scA_y$ is a localization of the $\cin$-ring
$\scA$ at the set $\{y\not =0\} := \{a\in \scA \mid y(a)\not =0\}$ and
$\pi_x:\scB \to \scB_x$ is the localization of $\scB$ at $\{x\not
=0\}$, it is enough to show that
\begin{equation} \label{eq:2.3}
  \varphi(\{x\circ \varphi\not =0\}) \subseteq
  \{x\not = 0\}.
\end{equation}
Since
\[
(x \circ \varphi)(a) \not = 0 \quad \Leftrightarrow \quad x ( \varphi(a)) \not =0,
\]
\eqref{eq:2.3} holds and we are done.
\end{proof}

\begin{construction}[Construction of $\Spec$ on morphisms] \label{construction:2}
Let $\varphi:\scA \to \scB$ be a map of $\cin$-rings.   We construct
$\Spec(\varphi): \Spec(\scB)\to \Spec(\scA)$, a map of local
$\cin$-ringed spaces.  First define a map
of sets $f(\varphi):X_\scB \to X_\scA$ (where as before
$X_\scA= \Hom (\scA,\R)$ and $X_\scB= \Hom(\scB, \R)$) by
\[
f(\varphi)(x):= x\circ \varphi.
\]
Recall that $\{ V_a = \{ y\in X_\scA \mid y(a) \not =0\}\}_{a\in
  \scA}$ is a basis for the topology on $X_\scA$ and similarly $\{ U_b
= \{ x\in X_\scB \mid x(b) \not =0\}\}_{b\in   \scB}$ is a basis for
the topology on $X_\scB$.  Since
\[
(f(\varphi))\inv (V_a) = \{ x \in X_\scB\mid f(\varphi) x \in V_a\} =
\{ x \in X_\scB \mid x (\varphi(a)) \not =0\} = U _{\varphi(a)},
\]
the map $f(\varphi)$ is continuous with respect to the Zariski
topologies on $X_\scB$ and $X_\scA$.  It remains to construct a map of
sheaves
\[
f(\varphi)_\#: \scO_\scA \to f(\varphi)_* \scO_\scB.
\]
Fix an open set $U\subseteq X_\scA$.  We construct a map 
\[
f(\varphi)_{\#, U}: \cP_\scA (U)\to \scO_\scB (f(\varphi)\inv U).
\]
of $\cin$-rings as follows.  Recall that
\[
  \cP_\scA (U) = \{ \frs_a: U\to \bigsqcup _{x\in U} \scA_x\mid a\in
  \scA\}
\]
where, as in \eqref{eq:2.6}, $\frs_a(y) = a_y$ for all $y\in U$.  Here
and below to reduce the clutter we write $\frs_a$ when we mean $\frs_a|_U$.
Given $\frs_a \in \cP_\scA (U)$ consider
\[
\ti{\frs}_a : f(\varphi)\inv (U) \to \bigsqcup _{x\in f(\varphi)\inv
  (U)} \scB_x, \qquad \ti{\frs}_a (x)  := \varphi_x (\frs_a (f(\varphi)x)),
\]
where $\varphi_x: \scA_{f(\varphi)x} \to \scB_x$ is the map from
Lemma~\ref{lem:2.25}.   Note that 
\[
\varphi_x (\frs_a (f(\varphi)x)) = \varphi_x (\pi_{f(\varphi) x} a)
= \pi_x (\varphi(a)),
\]  
where the last equality is commutativity of \eqref{eq:2.4n}.   Hence
$\ti{\frs}_a  = \frs_{\varphi(a)} \in \cP_\scB(f(\varphi)\inv (U))
\subset \scO_\scB(f(\varphi)\inv (U))$.  We therefore define
\[
  f(\varphi)_{\#, U}: \cP_\scA (U)\to \scO_\scB (f(\varphi)\inv U)
\]
by
\[
f(\varphi)_{\#, U} (\frs_a): = \frs_{\varphi(a)}
\]  
for all $a\in \scA$.  It is easy to see that the family of maps
$\{f(\varphi)_{\#, U}\}_{U\subset X_\scA}$ is a map of presheaves
$f(\varphi)_\#: \cP_\scA \to f(\varphi)_* \scO_\scB$.  By the
universal property of sheafification we get a map of sheaves
$f(\varphi)_\#: \scO_\scA \to f(\varphi)_* \scO_\scB$.

We define
\[
  \Spec(\varphi) \equiv
  \uu{f(\varphi)}: = (f(\varphi), f(\varphi)_\#).
\]  
\end{construction}

\begin{lemma} \label{lem:3.23} For any map $\varphi: \scA \to \scB$ of Poisson
  $\cin$-rings the induced map on spectra $\Spec(\varphi) \equiv (f(\varphi), f(\varphi)_\#):\Spec(\scB))
  \to \Spec(\scA)$ is Poisson.
\end{lemma}  

\begin{proof} It is enough to show that for any open subset $U$ of $X_\scB$
the map
\[
  f(\varphi)_{\#, U}: \cP_\scA (U)\to \scO_\scB (f(\varphi)\inv U),
  \qquad \frs_a\mapsto  \frs_{\varphi(a)}
\]
is Poisson.    Recall that the bracket on $\cP_\scA (U)$ is given by
\[
\{\frs_a, \frs_b\} = \frs_{\{a,b\}}
\]
for all $a, b\in \scA$ (as before we drop the restrictions to $U$ to reduce
the clutter).   Since $\varphi (\{a, b\}) = \{\varphi(a),
\varphi(b)\}$ the result follows.
\end{proof}  

\section{Applications} \label{sec:ex}

\subsection{Marsden-Weinstein-Meyer reduction} \label{subsec:sr}\quad 
Let $(M, \omega, \mu:M\to \fg^\vee)$ be a symplectic manifold with a
Hamiltonian action of a Lie group $G$ and associated equivariant
moment map $\mu$.  Let $\scA$ be the $\cin$-ring of $G$-invariant
functions $\cin(M)^G$.  As we observed in Example~\ref{ex:ACG}, for any point
$\alpha\in \fg^\vee$ the ideal $I_\alpha \subset \scA$ of invariant
functions that vanish on the level set $\mu\inv (\alpha)$ is a Poisson
ideal.  Hence $\Spec(\scA/I_\alpha)$ is a Poisson $\cin$-scheme.

In the case where the $\cin$-ring $\cin(M)^G$ is sufficiently nice and
the action of $G$ on $M$ is proper the space $X_{\scA/I_\alpha}$ of
points of the scheme $\Spec(\scA/I_\alpha)$ is
homeomorphic to $\mu\inv (\alpha)/G_\alpha$ (where $G_\alpha$ denotes
the stabilizer of $\alpha \in \fg^\vee $ under the coadjoint actions).
We'll discuss it in the next section, Section~\ref{sec:compute}.
Therefore we may regard the scheme $\Spec(\scA/I_\alpha)$ as (a
possibly singular) symplectic quotient (a.k.a.\ reduced space)
$M/\!/_\alpha G$.

If the action of $G$ on $M$ is {\sf not} proper, the space
$X_{\scA/I_\alpha}$ may be quite different from the topological
quotient $\mu\inv (\alpha)/G_\alpha$.  Here is an example.
\begin{example} \label{ex:4.1}
Consider
the action of $\R$ on the torus $\T^2 = \R^2/\Z^2$ by an irrational
flow: 
\[
t \cdot [(x,y)] := [x+t, y +\sqrt{2} t]
\]
 for
all $t\in \R$ and all $[x,y]\in \R^2/\Z^2$.  Then the lifted action of
$\R$ on the cotangent bundle $T^*\T^2 \simeq \T^2 \times \R^2$ is
Hamiltonian with the moment  $\mu([x,y], p_x, p_y) = p_x + \sqrt{2}
p_y$. Consequently
\[
\mu\inv (\alpha) = \{([x,y], p_x, p_y) \in T^*\T^2 \mid p_x = \alpha -
\sqrt{2} p_y\}
\]  
and
\[
\mu\inv (\alpha) /\R \simeq (\T^2/\R)\times \R.
\]
On the other hand the Poisson algebra of $\R$-invariant functions on
$T^*\T^2$ is $\cin(T^*\T^2) ^\R\simeq \cin(\R^2)$ with the zero
Poisson bracket and
\[
  I_\alpha \simeq \{f\in \cin(\R^2) \mid f(\alpha - \sqrt{2}p_y, p_y)
  = 0\textrm{ for all }p_y\in \R\} .
\]
Hence
\[
\cin(T^*\T^2) ^\R\simeq \cin(\R^2)/I_\alpha \simeq \cin(\R)
\]
(with the zero Poisson bracket) and
\[
X_{\cin(T^*\T^2) ^\R}\simeq X_{ \cin(\R^2)/I_\alpha} \simeq X_{\cin(\R)} = \R.
\]
Note that there is an evident map
\[
  \mu\inv (\alpha) /\R \simeq (\T^2/\R)\times \R \to
  X_{\cin(T^*\T^2)^\R}% 
  \simeq \R
\]
which  identifies all the points that cannot be separated by invariant
functions.
\end{example}
\subsection{Poisson reduction}\label{subsec:pr}\quad The construction of the Poisson
scheme in \ref{subsec:sr} does not require the assumption that the
manifold $M$ is symplectic --- we only need that $M$ is a
Poisson manifold, that the action of the group $G$ preserves the
Poisson bracket (and hence the $\cin$-subgring $\cin(M)^G$ of
invariant functions is Poisson) and that the Hamiltonian flows of
invariant functions preserve a closed subset $X\subset M$.  Then by
Lemma~\ref{lem:2.10} the vanishing ideal 
\[
	{\cI}(X) := \{f \in \cin(M)^G : f|_X = 0\}
\]
is a Poisson ideal in the $\cin$-ring $\cin(M)^G$ of invariant
functions.  Hence $\cin(M)^G/\cI(X)$ is a Poisson $\cin$-ring and
$\Spec(\cin(M)^G/\cI(X)$ is a Poisson $\cin$-scheme.

For example, the action of the group $G$ on the Poisson manifold $M$
may have an equivariant moment map $\mu:M\to \fg^\vee$ \cite{E}.   Then as in
the symplectic case the level sets $\mu\inv (\alpha)$, $\alpha \in
\fg^\vee$ are preserved by the Hamiltonian flows of invariant
functions and we may take the closed set $X$ to be a level set
$\mu\inv (\alpha)$, just as we did in the symplectic case.

Since $\varpi: \cin(M)^G \to \cin(M)^G/\cI(X)$ is surjective,
Lemma~\ref{lem:closed} implies that the
induced map $f(\varphi): X_{\cin(M)^G/\cI(X)} \to X_{\cin(M)^G}$ of
spaces is a closed embedding.  Furthermore, by Lemma~\ref{lem:5.8}, if
the action of $G$ on $M$ is proper then the space of points
$X_{\cin(M)^G}$ is homeomorphic to $M/G$.   Hence in the case of
proper actions the
space of points of the affine Poisson scheme $\Spec(\cin(M)^G/\cI(X))$
is homeomorphic to a closed subset of the orbit space $M/G$.

\subsection{quasi-Poisson reduction}\label{subsec:qpr}\quad
Let $(M, \mathsf{P}, \Phi: M\to G)$ be a quasi-Poisson manifold in the
sense of Alekseev, Kosmann-Scwarzbach and Meinrenken
\cite{AKSM}.  As explained in Section~6 of \cite{AKSM} the $\cin$-ring of invariant functions $\cin(M)^G$  is a Poisson
$\cin$-ring.   Therefore $\Spec( \cin(M)^G)$ is a Poisson
$\cin$-scheme.

It is shown in \cite{AKSM} in the course of the proof  of Theorem~6.1
that for any invariant function $f$ on $M$ the moment map $\Phi$ is
constant along the integral curves of the Hamiltonian vector field of
$f$.  Arguing as in the proof of Lemma~\ref{lem:2.10} we see that for
any conjugacy class $\scC \subset G$ the ideal
\[
I_\scC := \{f \in \cin(M)^G \mid f|_{\Phi\inv (\scC)} = 0\}
\]
is a Poisson ideal in the Poisson $\cin$-ring $\cin(M)^G$.
Consequenty the quotient $\cin(M)^G/I_\scC$ is a Poisson $\cin$-ring
and its spectrum $\Spec(\cin(M)^G/I_\scC)$ is a Poisson
$\cin$-scheme.  As before the quotient map
$\cin(M)^G\to \cin(M)^G/I_\scC$ induces a closed embedding
$X_{\cin(M)^G/I_\scC} \hookrightarrow X_{\cin(M)^G}$.    

\subsection{Schemes associated to various ideals in Poisson
  manifolds}\label{subsec:pi}\quad
Let $P$ be a Poisson manifold and $I\subset \cin(P)$ an ideal in the
$\cin$-ring $\cin(P)$.  By Lemma~\ref{lem:SW} the normalizer of $I$ in
the Poisson algebra $\cin(P)$ is a Poisson $\cin$-ring.  We now
consider two cases.\\

\noindent  {\bf Case a.}\quad The ideal $I$ is a {\em Poisson
  subalgebra } of $\cin(P)$, i.e, for all $f,g \in I$ the bracket
$\{f, g\}\in I$.  Then the ideal $I$ is contained in its normalizer
$N(I)$. Hence the quotient $N(I)/I$ is a Poisson $\cin$-ring and its
spectrum $\Spec(N(I)/I)$ is a Poisson $\cin$-scheme.  This is a version
of coisotropic reduction in Poisson geometry.

In more detail, the
inclusion $N(I) \hookrightarrow \cin(P)$ induces an injective map
$\varphi: N(I)/I \to \cin(P)/I$. Taking $\Spec$ we get a map of
$\cin$-schemes $\Spec(\varphi) : \Spec(\cin(P)/I) \to \Spec(N(I)/I)$.

Note that in general if $\psi: \scA\to \scB$ is an injective map of
$\cin$-rings, the induced map of points $f(\psi): X_\scB \to X_\scA$ need not
be surjective.  For example consider a smooth map $F:M\to N$ between
two manifolds with $F(M)$ dense in $N$.  Then $F^*:\cin(N) \to
\cin(M)$ is injective and $f(F^*) = F$  (we use the notation of
Construction~\ref{construction:2}), so $f(F^*)$ it's not surjective.

However, suppose $C\subset P$ is a
coisotropic submanifold and let $I$ denote the ideal of functions
that vanish on $C$.   Then $I$ is a Poisson subalgebra of $\cin(P)$
and  the Poisson $\cin$-ring $N(I)/I$  is
isomorphic to the $\cin$-ring of functions on $C$ that are constant
along the leaves of the null foliation $\scF$.  If the space of leaves
$C/\scF$ is a
manifold then $C/\scF)$ is isomorphic to $\Spec(N(I)/I)$.   In general we
view the Poisson $\cin$-scheme $\Spec(N(I)/I)$ as the coisotropic
reduction of $C$.  Note that the underlying topological space $X_{N(I)/I}$ may be
nicer than the leaf space $C/\scF$ since (as was mentioned above) the
space of points of a $\cin$-ring is Hausdorff and completely regular.
We have seen an instance of it when we looked at the reduction of the cotangent
bundle of the 2-torus $T^*\T$ by the irrational flow (Example~\ref{ex:4.1}).

\begin{example}[Sniatycki-Weinstein reduction] \label{ex:SW2}
We now revisit Example~\ref{ex:SW}.  Recall that one starts with a
symplectic manifold $(M, \omega)$ equipped with  a Hamiltonian action
of a connected Lie group $G$ and associated equivariant moment map
$\mu:M\to \fg^\vee$.  Then the ideal $\scI \subset \cin(M)$ generated
by the components of the moment map $\mu$ is closed under the Poisson
brackets, hence $N(\scI)/\scI)$ is Poisson $\cin$-ring equipped
with an injective map $\varphi: N(\scI)/\scI \to \cin(M)/\scI$.
Hence $\Spec(N(\scI)/\scI$ is a Poisson scheme which we may view as a
coisotropic reduction of $\cin(M)/\scI$.

Note that since the quotient map $q: \cin(M) \to \cin(M)/\scI$ is
surjective, the induced map
$f(q): X_{\cin(M)/\scI} \to X_{\cin(M)} \simeq M$ is a closed
embedding by Lemma~\ref{lem:closed}.  However, in general the ideal
$\scI$ may be strictly smaller than the vanishing ideal of
$\mu\inv(0)$ (cf.\ \cite{AGJ}) and the scheme $\Spec(\cin(M)/\scI)$ need not be reduced.
\end{example}

\noindent  {\bf Case b.}\quad The ideal $I$ is an arbitrary ideal of
the $\cin$-ring $\cin(P)$.  Then the intersection $N(I) \cap
I\subseteq N(I)$ is a Poisson ideal in $N(I)$.  Consequently
$\Spec(N(I)/\left(I\cap N(I))\right)$ is a Poisson scheme.  Note that the inclusion
$N(I) \hookrightarrow \cin(P)$ induces an injective map
\begin{equation}
\varphi: N(I)/(I\cap N(I)) \hookrightarrow \cin(P)/I.
\end{equation} 
If the map $\varphi$ above is an isomorphism then,  since $ N(I)/(I\cap
N(I))$ is Poisson  $\cin$-ring,  the quotient $\cin(P)/I$  is a also 
Poisson $\cin$-ring.   Consequently the close subscheme $\Spec(\cin(P)/I)$  of
the manifold $P$ is a Poisson scheme.   When $I$ is a vanishing ideal
of a submanifold $M\subset P$, then $\Spec(\cin(P)/I) = M$ and $M$ is
a Poisson-Dirac submanifold of $P$. 

 In general whenever $\varphi$ is
an isomorphism we can define $\Spec(\cin(P)/I)$  to be a Poisson-Dirac
subscheme $P$ (without any smoothness assumption).

If the map $\varphi$ fails to be surjective we still have a map
\[
\Spec(\varphi): \Spec (\cin(P)/I)\to \Spec N(I)/I\cap N(I)
\]
of local $\cin$-ringed spaces.   We can view the map as a kind of
coisotropic reduction of the closed subscheme $\Spec (\cin(P)/I)$ of
the Poisson manifold $P$.

\section{Computing $\Spec(\scA)$} \label{sec:compute}
Here is a brief summary of the content of this section.
We start by recalling what it means for a $\cin$-ring $\scA$ to be
finitely generated.   For such rings Dubuc \cite{Dubuc} gave a fairly
concrete description of $\Spec(\scA)$ which we  recall.   Next we
consider point determined $\cin$-rings (see Definition~\ref{def:A.2}).     We show that if $\scA$ is
point determined then the structure sheaf $\scO_\scA$ of $\Spec (\scA)$
is a sheaf of {\em functions} on the space $X_\scA$ of its points.
For finitely generated {\em and} point determined $\cin$-rings the description of
$\Spec(\scA)$ simplifies considerably: $\Spec(\scA)$ is a closed
subset of some $\R^N$ with the induced sheaf of ``smooth'' functions.

We then study $\Spec(\scA)$ where $\scA$ is a ring of invariant
functions $\cin(M)^G$ for an action of a Lie group $G$ on a manifold
$M$.  Such rings are point determined and there is always a map
$\overline{\ev}:M/G\to X_{\cin(M)^G}$ from the orbit space to the
space of points of $\cin(M)^G$.   If the invariant functions
separate points the map $\overline{\ev}$ is injective.   If
$\cin(M)^G$ contains a proper function then the map  $\overline{\ev}$
is surjective.  If the action of $G$ on $M$ is proper then
$\overline{\ev}$ is a homeomorphism.  And then the structure sheaf of
$\Spec(\cin(M)^G)$  is the sheaf of functions on the orbit space $M/G$
which has a familiar description.

Finally we consider the case where the action of $G$ on $M$ is proper
and the $\cin$-ring $\cin(M)^G$ is finitely generated.  This happens,
for example, when $G$ is compact and the action has finitely many
orbit type by a deep theorem of Schwarz \cite{Sch}.  Note that
finiteness of the number of orbit types is not necessary --- see
Example~6.9 in \cite{KL-toric}.   We then have two descriptions of
$\Spec(\cin(M)^G$: as a sheaf on functions on $M/G$ and as sheaf of
functions on a closed subset $Z$ of some $\R^n$.   Not surprisingly
the two descriptions a related in a way that should look familiar.  We
now get to work.\\

\subsection{$\Spec(\scA)$ for a finitely generated $\cin$-ring
  $\scA$}\quad \label{subsec:spec_fg}
Recall that for any natural number $n$ the $\cin$-ring $\cin(\R^n)$ is
freely generated by the set $\{x_1,\ldots, x_n\}$ of the standard
coordinate functions (see \cite{MR}, for example).  This is because
for any smooth function $f\in \cin(\R^n)$
\[
  f = f\circ (x_1,\ldots, x_n)  = f_{\cin(\R^N)} (x_1,\ldots, x_n).
\]  
\begin{definition}
A $\cin$-ring $\scA$ is {\sf finitely generated} if there is a
surjective map $\varpi: \cin(\R^n) \to \scA$ of $\cin$-rings.  And
then $\{\varpi(x_1), \ldots \varpi(x_n)\} \subset \scA$ is a set of generators of $\scA$.
\end{definition}
In \cite{Dubuc} on p.~687 in the paragraph right below Proposition~12 
Dubuc describes the affine scheme $\Spec(\scA)$ for a finitely
generated $\cin$-ring $\scA$ as follows. Since $\scA$ is finitely
generated there is a surjective map of $\cin$-rings $\Pi: \cin(\R^n)
\to \scA$.  Let $J$ denote the kernel of $\Pi$ and let
\begin{equation} \label{eq:A.vanish}
Z = Z_J:= \{p\in \R^n \mid f(p) = 0 \textrm{ for all } f\in J\}
\end{equation}
denote the zero set of the ideal $J$.   For each open set $U\subset
\R^n$ we then have the quotient $\cin$-ring
\[
\cP(U) = \cin(U)/J_U
\]  
where $J_U \subset \cin(U)$ is the ideal generated by the set $J|_U =\{f|_U\mid f\in
J\}$.  Note that if $U\cap Z_J= \varnothing$, then $\cP(U) = 0$.  Let
$\cE$ be the sheafification of $\cP$.  The support of $\cE$ is $Z_J$,
hence we can view $\cE$ as a sheaf on $Z$.   Dubuc then tells us that  $\Spec(\scA) \simeq
(Z, \cE)$.

\subsection{$\Spec(\scA)$ for a point determined $\cin$-ring
  $\scA$}\quad  It will be useful for us to formulate the definition
of a point determined $\cin$-ring as follows (it is equivalent to
Definition~4.1(a) in \cite[p.\ 44]{MR}).

\begin{definition} \label{def:A.2}
A $\cin$-ring $\scA$ is {\sf point determined} if for every element
$a\in \scA$
\[
\big( x(a) = 0 \textrm{ for all }x\in \Hom(\scA, \R) \equiv X_\scA \big)\quad \Rightarrow a = 0.
\]
Equivalently the map
\[
\scA \to \prod_{x\in X_\scA} \R, \quad a\mapsto (x(a))_{x\in X_\scA}
\]
is injective.
\end{definition}  
\begin{remark}
There are many $\cin$-rings that are not point determined. For example the quotient $\cin$-ring $\cin(\R)/(x^2)$ is {\em not} point determined: there is only one $\R$-point induced by evaluation at $0$, but  $\cin(\R)/(x^2)$ has dimension 2 as a real vector space, so cannot be isomorphic to $\R$.
\end{remark}
Here is another interpretation of what it means for a $\cin$-ring
$\scA$ to be point determined which is more geometrically
intuitive. For every element $a\in \scA$ we have a function $a_*:
X_\scA\to \R$, $a_* (x):= x(a)$.   The functions $a_*$ are continuous
with respect to the Zariski topology  $\scT_\scA$ on 
  $X_\scA$ and, moreover, $\scT_\scA$ is the smallest topology making
  all the funtions $a_*$ continuous \cite{Joy}.   In general the map
\begin{equation} \label{eq:A.1}
\scA \to C^0(X_\scA), \quad a\mapsto a_*
\end{equation}
may not be injective.  It is injective if and only if $\scA$ is point
determined.\\

\noindent
We now argue that if a $\cin$-ring $\scA$ is point determined
then the structure sheaf $\scO_\scA$ of the affine $\cin$-scheme
$\Spec(\scA)$ is a sheaf of continuous functions on its space $X_\scA$
of points.

\begin{lemma} \label{lem:A.3}
Let $\scA$ be a point determined $\cin$-ring.  Define a presheaf $\cP$
on the space $X_\scA$ of points of $\scA$  by
\begin{equation} \label{eq:A.3}
\cP(U) = \{ a_*|_U \mid a\in \scA\}
\end{equation}
for all $U\in \Open(X_\scA)$ (where as before $a_*(x) := x(a)$ for all
$a\in \scA$ and all $x\in X_\scA$) . Then the structure sheaf $\scO_\scA$
of $\Spec(\scA)$ is isomorphic to the sheafification of the presheaf
$\cP$.
\end{lemma}
  
\begin{proof}
Recall that the structure sheaf $\scO_\scA$ is the sheafification of the
presheaf $\cP_\scA$ given by
\[
\cP_\scA(U) = \{\frs_a|_U \mid a\in \scA\} 
\]  
(cf.\ \eqref{eq:sec}) and that
$\cP_\scA(U)\simeq \scA/\bigcap_{x\in U}I_x$.  We construct a map of
presheaves $\psi:\cP_\scA \to \cP$ and prove  that $\psi$ is an
isomorphism.

Given $U\in \Open(X_\scA)$ consider the map
\[
\mu_U: \scA \to \cP(U), \quad \mu_U(a) := a_*|_U.
\]
Suppose $a\in \bigcap _{x\in U}I_x$.  Then for any $x\in U$, we have
$a\in I_x$.  This means that there is $d\in \scA$ so that $x(d)\not
=0$ and $ad=0$.  Since $0\not = x(d) = d_* (x)$ and $d_*$ is
continuous there is a neighborhood $V$ of $x$ so that $d_*(y) \not =0$
for all $y\in V$.   Since $ad = 0$, $0 = (ad)_* = a_* d_*$.  Hence
$a_*|_V = 0$.  Since $x\in U$ is arbitrary, $a_*|_U = 0$.  Therefore
\[
\psi_U: \cP_\scA(U)\simeq  \scA/\bigcap_{x\in U} I_x \to \cP(U),
\quad \frs_a|_U\mapsto a_*|_U
\]  
is a well-defined map of $\cin$-rings.  It is not hard to check that
the collection of maps $\{\psi_U\}_{U\in \Open(X_\scA)}$ assemble into a map of presheaves $\psi: \cP_\scA \to \cP$.

It remains to show that the maps $\psi_U$ are isomorphism.  It's clear
from the construction that $\psi_U:\cP_\scA(U)\to \cP(U)$ is onto for
every open set $U$.  Suppose $\frs_a|_U\in \ker (\psi_U)$.  Then
$a_*|_U = 0$.  Since the topology on $X_\scA$ is the smallest topology
such that all functions $a_*$ ($a\in \scA)$) are continuous, for every
open set $V\subset X_\scA$ for every $x\in V$ there is $d\in \scA$ so
that $d_*(x)\not = 0$  and $\supp (d_*)\subset V$ (see, for example
\cite[Lemma~2.15]{KL}).  Then
\[
0=  a_* d_* = (ad)_*
\]
Since the map \eqref{eq:A.1} is injective
($\scA$ is point determned), it follows that $da =0$   Since $x(d) =
x_*(d) \not =0$ it follows that $a\in I_x$.   Since $x\in U$ is
arbitrary, $a\in \bigcap_{x\in U}I_x$.  Therefore $\frs_a|_U = 0$ and
$\psi_U$ is injective.
\end{proof}  
We now describe $\Spec(\scA)$ when the $\cin$-ring $\scA$ is finitely
generated {\em and} point determined.

\begin{lemma}
Let $J\subset \cin(\R^n)$ be an ideal and 
\[Z= Z_J = \{p\in \cin(\R^n)\mid f(p) = 0 \textrm{ for all } f\in J\} 
\]
is its vanishing set and 
\[
I(Z_J) = \{f\in \cin(\R^n) \mid f(p) = 0 \textrm{ for all } p\in Z_J\}
\]
the zero ideal of $Z_J$.

 If the quotient $\scA := \cin(\R^n)/J$ is a point determined $\cin$-ring then
 \[
 I(Z_J) = J.
\] 
\end{lemma}
\begin{proof}
Since $J\subset I(Z_J)$ we only need to show that $I(Z_J) \subset J$.

Recall that for every point $p\in \R^n$ we have an $\R$-point $\ev_p: \cin(\R^n)\to \R$ 
and that the map $\ev: \R^n \to \Hom (\cin(\R^n), \R)$ is a bijection (cf.\ Remark~\ref{rmrk:Milnors}).  The maps $\ev_p$ factor through $\cin(\R^n)/J$ if and only if $p\in Z_J$.
We thus have an induced map 
\begin{equation} 
\overline{\ev}:Z_J \to \Hom (\cin(\R^n/J), \R), \quad \overline{\ev}_p (f+J) = f(p).
\end{equation}
If $x: \cin(\R^n)/J\to \R$ is an $\R$-point then $x \circ \Pi :\cin(\R^n) \to \R$ is also an $\R$-point (where $\Pi:\cin(\R^n) \to \cin(\R^n)/J$ is the quotient map).
Consequently there is a unique $p\in \R^n$ so that 
\[
x\circ \Pi =\ev_p.
\]
Since $(x \circ \Pi)\,(f) = 0$ for all $f\in J$, such point $p$ must lie in $Z_J$.  Therefore $\overline{\ev}:Z_J \to \Hom (\cin(\R^n/J), \R)$ is surjective, i.e.,
\[
\Hom (\scA, \R)  = \{\overline{ev}_p \mid p\in Z_J\}.
\]
Suppose now $f\in I(Z_J)$.  Then
\[
0 = f(p) = \ev_p(f) = \overline{\ev}_p(f+J)
\]
for all $p\in Z_J$.  Hence $x(f+J) = 0$ for all $x\in \Hom (\scA,
\R)$.  Since $\scA$ is point determined,  $f+J = 0$ in $\scA = \cin(\R^n)/J$.
Thus $f\in J$.
\end{proof}

It follows from \ref{subsec:spec_fg} that if $\scA = \cin(\R^n)/J$ is
point determined then $X_\scA = Z_J$ and $\scO_A $ is given by
\[
\scO_\scA(U)  = \cin(\widetilde{U})|_U
\]  
for $U\in \Open(Z_J)$.  Here $\widetilde{U} \subset \R^n$ is any open
set with $\widetilde{U} \cap Z_J = U$.

The fact that the space of points $X_{\cin(\R^n)/J}$ embeds into
$\R^n$ as a closed subspace is not a coincidence.

\begin{lemma} \label{lem:closed}
Let $\varpi:\scB \to \scA$ be a {\em surjective } map of $\cin$-rings.
Then the induced map on spaces $
f(\varpi) :X_\scA \to X_\scB$, $f(\varphi) x := x\circ \varpi$ is a closed embedding.
\end{lemma}

\begin{proof}
For an ideal $I\subset \scB$ define the set of its zeros to be
\[
Z_I:= \{y\in X_\scB \mid y (f) = 0 \textrm{ for all } f\in I\}.
\]
Since \[
Z_I = \bigcap _{f\in I} \{x\in \scB \mid x(f) =0\}, 
\]
the set $Z_I$ is closed in
the set of points $X_\scB$.
We know that for any map of $\cin$-rings $\varphi:\scB\to
\scA$ the induced map $f(\varphi):X_\scA \to X_\scB$, $f(\varphi) x = x \circ \varphi$, is
continuous.

Let $J:= \ker \varpi$. By the first isomorphism theorem
\[
f(\varpi) \big( X_\scA \big)= \{y\in X_\scB \mid y|_J = 0\},
\]  
which is $Z_J$.  Thus $f(\varpi)$ maps $X_\scA$ continuously onto $Z_J$.
Since $\varpi$ is onto, $f(\varpi)$ is injective. To prove that
$f(\varpi) :X_\scA\to Z_J$ is a homeomorphism, it's enough to prove that
$f(\varpi)$ is an open map.  A basic open set in $X_\scA$ is of the form
$\{a \not = 0\}: = \{ x\in X_\scA \mid x(a) \not = 0\}$.   Since $\varpi$
is onto, there is $b\in \scB$ such that $\varpi(b) = a$.     Then
\[
x (\varpi(b)) \not = 0 \quad \Leftrightarrow \quad (x \circ \varpi) (b) \not =0.
\]
Consequently
\[
f(\varpi)\, (\{ a \not = 0\}) = \{ y \in X_\scB \mid  y = f(\varpi) x \,
\textrm{ and } \,  x(a) \not = 0\} = Z_J \cap \{ b \not =0\}, 
\]
and we are done.
\end{proof}

\subsection{ The spectrum of a $\cin$-ring  $\Spec(\cin(M)^G)$ of 
  invariant functions.} % 
  \quad 
\label{subsec:A.inv} We now turn to study $\Spec(\scA)$
where $\scA$ is a ring of invariant functions $\cin(M)^G$ for an
action of a Lie group $G$ on a manifold $M$.

\begin{lemma} \label{lem:A.6}
Let $G\times M\to M$ be an action of a Lie group $G$ on a manifold
$M$. Then the $\cin$-ring $\cin(M)^G$ of the invariant functions is
point determined.
\end{lemma}  
  
\begin{proof}
For a point $p\in M$ we have a homomorphism
\[
\ev_p: \cin(M)^G\to \R, \quad \ev_p(f) : = f(p)
\]  
If $f\in \cin(M)^G$ and $x(f) = 0$ for all $x\in \Hom(\cin(M)^G, \R)$,
then $0 =\ev_p(f) = f(p)$ for all $p\in M$. Therefore  $f=0$.
\end{proof}

Let $G\times M \to M$ be an action.  If two points $p,p'\in M$ lie on
the same $G$-orbit than for any invariant function $f\in \cin(M)$
\[
\ev_p(f) \equiv f(p) = f(p') \equiv \ev_{p'}.
\]
Hence the evaluation map
\[
  \ev :M\to X_{\cin(M)^G}, \quad p\mapsto \ev _p
\]
descends to a well-defined map
\begin{equation} \label{eq:A.evbar}
\overline{\ev}:M/G\to X_{\cin(M)^G}, \quad \overline{\ev} (G\cdot p)
= \ev_p.
\end{equation}
In general there is no reason for the map $\overline{\ev}$ 
to be a bijection.  If invariant functions separate orbits (which
happens when the action is proper) then $\overline{\ev}$ is
injective. In general it is not.  For instance consider the action of
the reals $\R$ on the torus $\T^2$ by irrational flow (cf.\
Subsection~\ref{subsec:sr}).  Then the ring of invariants
$\cin(\T^2)^\R$ consists of constant functions.  Consequently it has
exactly one $\R$ point and the map
$ \overline{\ev}:\T^2/\R \to X_{\cin(\T^2)^\R}$ is not injective.  We
now argue that for proper actions the map $\overline{\ev}$ of
\eqref{eq:A.evbar} is a homeomorphism.  We do it in a series of
lemmas.

\begin{lemma} \label{lem:proper}
Let $M$ be a topological space and $\scA \subset C^0(M)$ a $\cin$-subring.
Suppose $\scA$ contains a proper function.
Then any unital  $\R$-algebra map $\varphi: \scA\to \R$
is an evaluation at a point $p\in M$:
\[
\varphi(g) = g(p) \quad \text{ for all $g\in \scA$}.
\]
Consequently, all unital $\R$-algebra maps from $\scA$ to $\R$ are morphisms 
of $\cin$-rings and 
the map
\[
M \to \Hom(\scA, \R), \quad x \mapsto \ev_x
\]  
is surjective.   Here as before $\ev_x(g) : = g(x)$.
\end{lemma}

\begin{proof}  \footnote{I am greatful to Yael Karshon for
fixing a gap in an earlier version of the proof.}
  We use the same notation for a real number and for the 
  corresponding constant function in $C^0(M)$.  Since $\scA$ is a
  $\cin$-subring of $C^0(M)$, it contains these constant functions.
  
Let $\varphi  : \scA \to \R$ be a unital homomorphism of $\R$-algebras.
For $g\in \scA$ set 
\[
  U_g := \{ p \in M \mid g(p) \neq \varphi(g)\}  = g\inv (\R\smallsetminus \varphi(g)).
\]
The sets $U_g$ are open.
We now argue by contradiction. Assume that for every $p \in M$ there
exists $g \in \scA$ such that $\varphi(g) \neq g(p)$, i.e., $p\in U_g$.  Then
$\{U_g\}_{g\in \scA}$ is an open cover of $M$.  Let $f \in \scA$ be a
proper function, which exists by our assumption.  Let
\[
  r := \varphi(f).
\]  
Since  $f$ is proper, the level set $\{f = r \}$ is compact.
Therefore there exists $n>0$ and  $g_1,\ldots,g_n \in \scA$ such that 
\[
  \{f=r\}  \subseteq U_{g_1}\cup \ldots \cup U_{g_n}.
\]  
Then the  function
\[
  h(\cdot) := \big(f(\cdot)-r \big)^2 
   + \big(g_1(\cdot)-\varphi(g_1) \big)^2 
   + \ldots + \big(g_n(\cdot)-\varphi(g_n) \big)^2 \in \scA
 \]  
 is strictly positive.  
Since $\varphi$ is a homomorphism of $\R$-algebras, $\varphi(c) =c$ for any constant function $c$ hence $\varphi(\varphi(a)) = \varphi(a)$ for any $a\in \scA$.  Therefore 
\[
\varphi(h) = (\varphi(f) -r)^2 + (\varphi(g_1) -\varphi(\varphi(g_1)))^2 +\cdots +(\varphi(g_n) -\varphi(\varphi(g_n)))^2 = 0^2+ 0^2+\cdots +0^2 = 0.
\]
%$\varphi(h)=0$.
Since $h \geq (f(\cdot)-r)^2$ and since $(f(\cdot)-r)^2$ is proper and non-negative, 
the function $h$ is proper as well. 
Since $h$ is  proper and since $\R$ is locally compact Hausdorff, $h$
is a closed map.  It follows, since $h$ is positive, that there is
$\varepsilon >0$ so that $h(m) > \varepsilon$ for all $m\in M$.
There  exists a smooth function $k \in \cin(\R)$
such that
\[
  k(y) = \frac{1}{y} \quad \text{ for all \ } y> \varepsilon.
\]
Note that $k\circ h\in \scA$ since $\scA$ is a $\cin$-subring of
$C^0(M)$ and the $\cin$-ring operations on $C^0(M)$ are given by composition.
Since $h \cdot (k \circ h) \equiv 1$
and $\varphi$ is an $\R$-algebra morphism,
$ \varphi(h) \cdot \varphi(k \circ h) = \varphi(1) $.
But $\varphi(h) = 0$, so $\varphi(1) = 0$, which contradicts that
$\varphi$ is unital.
\end{proof}

\begin{remark} \label{rmrk:inv}
Let  $G\times M\to M$ be an action of a Lie group $G$ on a manifold
$M$. Let $q:M\to M/G$ denote the orbit/quotient map (in the category
of topological spaces).  By the
universal property of $q$ for any $G$-invariant function $f:M\to \R$
there exists a unique continuous function $\overline{f}:M/G \to \R$ with $f =
\overline{f}\circ q$.  We thus get an injective map
\begin{equation} \label{eq:Upsilon}
\Upsilon: \cin(M)^G \to C^0 (M/G), \qquad f\mapsto \overline{f}
\end{equation}
of $\cin$-rings.
We define
\begin{equation} \label{eq:5.1}
\cin(M/G): = \Upsilon(\cin(M)^G) \qquad 
\end{equation}
Then 
\[
\Upsilon: \cin(M)^G \to  \cin(M/G)
\]
is an isomorphism of $\cin$-rings.
\end{remark}

\begin{lemma} \label{lem:proper2}
Let $G\times M\to M$ be a proper action of a Lie group
$G$ on a manifold $M$. Then the $\cin$-subring $\cin(M/G)$ of $C^0(M/G)$ defined in
Remark~\ref{rmrk:inv} contains a proper function. 
Consequently the map 
\[
 \overline{\ev}: M/G \to X_{\cin(M/G)}, \quad  \overline{\ev} (G\cdot p) (f) = f(p)  \textrm{ for all } f\in \cin(M)^G
\]
is a bijection.
\end{lemma}  

\begin{proof}
As was mentioned before, since functions invariant under proper actions separate orbits, the
  map $\overline{\ev}$ is injective.

In the course of the proof of Theorem~4.3.1 in \cite{Palais} Palais
constructs a sequence of invariant functions $\{f_n\}_{n\in \N}
\subset \cin(M)^G$ with the property that their supports form a
locally finite cover of $M$, that $f_n(M)\subset [0,1]$ for all $n$,
$\sum f_n = 1$ and that sets $q(\supp (f_n)) \subset M/G$ are compact
(here as before $q:M\to M/G$ is the quotient map).  Let
$\overline{f}_n$ denote the image of $f_n$ in $C^0(M/G)$ under the map
$\Upsilon: \cin(M)^G \to C^0 (M/G)$ (cf.\ Remark~\ref{rmrk:inv}).
Then the function
\[ 
h = \sum_{n=1}^\infty n \overline{f}_n
\] 
is a desired proper function (cf.\ proof of Proposition~2.28 in
\cite{Lee}).

Hence for any $\varphi\in X_{\cin(M/G) } \simeq X_{\cin(M)^G}$ there is
  an orbit $G\cdot m\in M/G$ so that
  $\varphi(\overline{f}) = \overline{f}(G\cdot m) = f(m)$ for all
    $f\in \cin(M)^G$.  Here as before $\overline{f} = \Upsilon(f)$.
    Hence the map $\overline{\ev}$ is onto.  We have already observed
    that since the action of $G$ on $M$ is proper, the map
    $\overline{\ev}$ is injective.  Thus $\overline{\ev}$ is a
    bijection.
\end{proof}

\begin{lemma} \label{lem:5.8} Let $G\times M\to M$ be a proper action
  of a Lie group $G$ on a manifold $M$.  Then the space of points
  $X_{\cin(M)^G}$ of the $\cin$-ring of invariant functions is
  {\em homeomorphic} to the orbit space $M/G$.
\end{lemma}
  
\begin{proof} As we observed in Remark~\ref{rmrk:inv} the map $\Upsilon: \cin(M)^G\to
\cin(M/G)$ is an isomorphism of $\cin$-rings, hence $X_{\cin(M)^G}$ and
$X_{\cin(M/G}$ are homeomorphic. By 
Lemma~\ref{lem:proper2} %
the evaluation map
\[
\overline{\ev}: M/G \to X_{\cin(M)^G}, \quad \overline{\ev} (G\cdot
 p) (f) = f(p) \textrm{ for all } f\in \cin(M)^G%
 \]
is a bijection.   We need to show that $\ev$ is a homeomorphism.  By
Lemma~\ref{lem:5.10} below,  the topology on $M/G$ is generated by the sets
of the form $\{h\not = 0\} := \{p\in M/G \mid h(p)\not = 0\}$, $h\in \cin(M/G)$.  On the other hand the
topology on $X_{\cin(M/G)}$ is generated by the sets of the form
\[
\{f\not = 0\}:= \{x\in
X_{\cin(M/G)} \mid x(f) \not = 0\}, 
\]
$f\in \cin(M/G)$.  Now observe that
\[
  (\ev)\inv (\{x\in  X_{\cin(M/G)} \mid x(f) \not = 0\}) = \{ p\in
  M/G \mid f(p) \not = 0\}.
\]
Hence $\ev$ is a homeomorphism.
\end{proof}

\begin{lemma} \label{lem:5.10}
Let $G\times M\to M$ be a proper action of a Lie group $G$ on a
manifold $M$, $q:M\to M/G$ the quotient map and $\cin(M/G) \subset
C^0(M/G)$ defined by \eqref{eq:5.1}. The quotient topology on
$M/G$ is the smallest topology making all the functions in $\cin(M/G)$
continuous.  That is, the quotient topology is generated by the sets of the form $\{h\not
=0\}$ for all $h\in \cin(M/G)$.
\end{lemma}

\begin{proof}
By \cite{KL}[Lemma~2.15] it is enough to show that for any point $p\in
M/G$ and for any closed set $C\subset M/G$ with $p\not \in C$ there is
a function $f\in \cin(M/G)$ with $\supp(f) \cap C = \varnothing$ and
$f$ identically 1 near $p$.

To prove the existence of $f$ choose a point $m\in q\inv (p) \subset
M$.  Let $K$ denote the stabilizer of $m$.  Since the action of $G$ is
proper there is a slice $\Sigma$ through $m$ for the action.   Then
$G\cdot \Sigma$ is open in $M$, is diffeomorphic to $(G\times
\Sigma)/K$ and $(G\cdot \Sigma)/G \simeq \Sigma/K$ as topological
spaces. Consequently $q (G\cdot \Sigma)$ is an open neighborhood of
$p$ in $M/G$ and the composite $\Sigma\hookrightarrow M
\xrightarrow{q} M/G$ descends to an open embedding $q': \Sigma/K \to
M/G$.  On $\Sigma$ choose a smooth $K$-invariant function $\rho$ so
that $\rho$ is identically 1 near $m$ and $\supp (\rho)$ is contained
in a compact set which is disjoint from $(q')\inv (C)$.  Since $\rho$
is $K$-invariant, it descends to a continuous function $\tau$ on
$\Sigma/K\simeq (G\cdot \Sigma)/G\hookrightarrow M/G$.  Extend $\tau$
by 0 to all of $M/G$.  The resulting function is the desired function $f$.
\end{proof}

\begin{lemma}
Let $G\times M\to M$ be a proper action of a Lie group $G$ on a
manifold $M$ and let $q:M\to M/G$ denote the quotient map as before..  Then the affine $\cin$-scheme $\Spec(\cin(M^G))$ of the
ring of invariant functions is isomorphic to $(M/G, \cin_{M/G})$ where 
\[
\cin_{M/G} (U) = \{ f:U\to \R\mid f\circ q \in \cin(q\inv (U))\}
\]
for all open sets $U\subset M/G$.
\end{lemma}

\begin{proof}
By Lemma~\ref{lem:5.8} the space of points $X_{\cin(M)^G}$ is
homeomorphic to the orbit space $M/G$ by way of  the map
\[
\overline{\ev}: M/G\to X_{\cin(M)^G}, \quad
\left(\overline{\ev}(G\cdot p) \right) (f) = f(p).
\]  
Consequently
\begin{equation} \label{eq:A.7}
(\overline{\ev})^*: C^0 (X_{\cin(M)^G} )\to C^0(M/G), \quad h\mapsto
  h\circ \overline{\ev}
\end{equation}  
is an isomorphism of $\cin$-rings.

By Lemma~\ref{lem:A.6} the $\cin$-ring $\cin(M)^G$ is point
determined, hence can be identified with a $\cin$-subring of
$C^0(X_{\cin(M)^G})$ by way of
\begin{equation} \label{eq:A.8}
(\,\,)_*: \cin(M)^G \to C^0(X_{\cin(M)^G}), \quad f\mapsto f_*
\end{equation}
where $f_*(x) := x (f)$ for all $f\in \cin(M)^G$ and all $x\in
X_{\cin(M)^G}$.  Since for any $f\in \cin(M)^G$
\[
((\overline{\ev})^* f_*) (G\cdot p) = \left(\ev(G\cdot p)\right) (f) = f(p)
\]  
the composite of \eqref{eq:A.8} and \eqref{eq:A.7} is the map
$\Upsilon: \cin(M)^G \to C^0(M/G)$ that sends the invariant function
$f\in \cin(M)^G$ to the unique function $\overline{f}:M/G\to \R$ with
$\overline{f} \circ q = f$ (cf.\ Lemma~\ref{rmrk:inv}).
Lemma~\ref{lem:A.3} now implies the desired result.
\end{proof}  

\subsection{$\Spec(\cin(M)^G)$ when the action $G\times M\to M$ is
  proper and $\cin(M)^G$ is finitely generated.} \label{subsec:inv_fg}\mbox{}\\
We now put together Subsections~\ref{subsec:spec_fg} and \ref{subsec:A.inv}.
Suppose now $G\times M\to M$ is a proper action and the $\cin$-ring of
invariants $\cin(M)^G$ is finitely generated (hence there is a surjective
map $\varpi: \cin(\R^N) \to \cin(M)^G$ for some $N$).  Then, one one
hand, 
\[
\Spec (\cin(M)^G ) \simeq (Z, \cin_Z)
\]  
where as before
\[
Z = \{p\in \R^N\mid f(p) = 0 \textrm{ for all } f\in \ker\varpi\}
\]
is the zero set of the kernel of $\varpi$ and the sheaf $\cin_Z$ is
given by
\[
\cin_Z(U) = \cin(\widetilde{U} )|_U
\]  
for any $\widetilde{U}\in \Open (\R^n)$ with $\widetilde{U} \cap Z =
U$.
On the other hand 
\[
\Spec (\cin(M)^G)  \simeq (M/G, \cin_{M/G}).%
\]
By composing the two isomorphisms of local $\cin$-ringed spaces we get
an isomorphism
\[
\uu{\varphi} = (\varphi, \varphi_\#): (M/G, \cin_{M/G}) \to (Z, \cin_Z).
\]
We now write down $\uu{\varphi}$ explicitly.  Let $x_i:\R^N\to \R$,
$1\leq i\leq N$,
denote the coordinate functions. Then their images $f_i =
\varpi(x_i)$, $1\leq i\leq N$, generate $\cin(M)^G$.  Let 
\[
F = (f_1,\ldots, f_N):M\to \R^N.
\]  
Then
\[
g\circ F = \varpi(g)
\]  
for all $g\in \cin(\R^N)$, i.e., $\varpi = F^*$.  The closed embedding
\[
\Spec(\varpi): (X_{\cin(M)^G}, \scO_{\cin(M)^G}) \simeq (M/G,
\cin_{M/G} \to \Spec(\cin(\R^N) = (\R^N, \cin_{\R^N})
\]  
is then $\Spec(\varpi) = (\overline{F}, F_\#)$ where
$\overline{F}:M/G \to \R^N$ is the continuous map induced by $F:M \to
\R^N$ (i.e., $\overline{F}(G\cdot p) = F(p)$) and $F_\#: \cin_{\R^N}
\to \cin_{M/G}$ is
\[
F_{\#, W}: \cin(W) \to \cin_{M/G} ((\overline{F})\inv (W)0 \simeq
\cin(F\inv (W)^G, \quad h\mapsto h\circ F
\]
for all open sets $W\subset \R^N$.  Then
\[
\overline{F}(M/G) = F(M) = Z
\]  
and $\ker(F_{\#})$ is the ideal sheaf of $Z$.  Thus $\Spec(\varpi)$
factors through
\[
(\varphi, \varphi_\#): (M/G, \cin_{M/G}) \to (Z, \cin_Z)
\]  
where $\varphi (G\cdot p) = F(p)$ and
\[
\varphi_{\#, W} : \cin_Z(W) \to \cin(F\inv (W)^G, \quad \varphi_{\#,
  W} (h)  = h\circ F.
\]  
\mbox{}

\section{A remark on the Cushman-Weinstein conjecture} \label{sec:CW}

Suppose $M$ is a symplectic manifold with a proper Hamiltonian action
of a Lie group $G$ and an associated equivariant moment map $\mu:M\to
\fg^\vee$.  Suppose further that the $\cin$-ring $\cin(M)^G$ is
finitely generated.    Then the $\cin$-ring of invariants $\cin(M)^G$
is Poisson and consequently $\Spec(\cin(M)^G)\simeq (M/G, \cin_{M/G})$
is a Poisson $\cin$-scheme.  Moreover as we saw in \ref{subsec:inv_fg}
a choice of a surjective map
$\varpi:\cin(\R^N) \to \cin(M)^G$ (i.e., a choice of generators of
$\cin(M)^G$) gives rise to a closed embedding
\[
\Spec(\varpi): (M/G, \cin_{M/G}) \to (\R^N, \cin_{\R^N}),
\]  
which induces an isomorphism
\[
(M/G, \cin_{M/G})  \to (Z, \cin_Z)
\]
for an appropriate closed subset  $Z $ of $ \R^N$. 
In particular
the $\cin$-ring
\[
\cin_Z(Z) = \cin(Z)=  \cin(\R^N)|_Z
\]  
carries a Poisson bracket.   In 1980s Cushman \cite{SL} and,
independently, Weinstein \cite{Eg-th, Eg-era} conjectured that the
Poisson bracket on $\cin(Z) $ extends to a Poisson bracket on
$\cin(\R^N)$.   Egilsson \cite{Eg-th, Eg-era} showed that the
conjecture is false as stated.  Further work in this direction was
done by Davis \cite{Davis}  and McMillan \cite{Mc}.  In this section
we make the following observation.

\begin{lemma}
Let $q: \cin(\R^N)\to \scB$ be a surjective map of $\cin$-rings.
Suppose $\scB$ is Poisson with the bracket $\{\cdot, \cdot\}$. There
exists a bivector field $\Lambda$ on $\R^N$ such that
\[
q(\langle df \wedge dg, \Lambda \rangle ) = \{q(f), q(g)\}
\]
for all $f, g\in \cin(\R^N)$.
\end{lemma}

\begin{proof}
Recall that the $\cin$-ring $\cin(\R^N)$ is freely generated by the
coordinate functions $x_1,\ldots, x_N:\R^N\to \R$.  More concretely,
for any $f\in \cin(\R^N)$ we have
\[
f = (f)_{\cin(\R^N)}(x_1, \ldots, x_N).
\]  
Since by assumption $q:\cin(\R^N)\to \scB$ is surjective, for all
indicies $i$ and $j$ there exist $c_{ij}\in \cin(\R^N)$ so that
$q(c_{ij}) = \{q(x_i), q(x_j)\}$.  Define the bivector field $\Lambda$
by
\[
\Lambda = \frac{1}{2}\sum_{i,j} c_{ij} \partial _i \wedge \partial_j.
\]  
Then
\[
  \langle df \wedge dg, \Lambda \rangle = \sum _{i,j} c_{ij}
  \partial_i f \,\partial_j g.
\]
Consequently
\[
 \begin{split} 
  q(\langle df \wedge dg, \Lambda \rangle)& =\sum _{ij}q( c_{ij})
  q(\partial_i f )\,q(\partial_j g) \\
&= \sum _{ij} (\partial_i f)_\scB (q(x_1),\ldots, q(x_N) (\partial_j
g)_\scB (q(x_1),\ldots, q(x_N) \{q(x_i) , q(x_j)\}\\
&= \{ f_\scB (q(x_1),\ldots, q(x_N) , g_\scB (q(x_1),\ldots, q(x_N) \}
= \{q(f), q(g)\}
\end{split}
\]
for all $f, g\in \cin(\R^N)$.

\end{proof}  

It follows that if $(M, \omega, \mu:M\to \fg^\vee)$ is a Hamiltonian $G$-space, the action of $G$ on $M$ is proper and the ring of invariants $\cin(M)^G$ is finitely generated with generators $f_1, \ldots f_N$ then 
\[
F:M\to \R^N, \quad F(m) = (f_1(m), \ldots f_N(m))
\]
induces an embedding $\overline{F}:M/G\to \R^N$ so that $Z = \overline{F}(M/G)$ is closed in $\R^N$ and ${\overline{F}}^* :\cin(Z) \to \cin(M/G)$ is an isomorphism of $\cin$-rings.
Consequently $\cin(Z) = \cin(\R^N)|_Z$ is a Poisson $\cin$-ring, and
the Poisson bracket on $\cin(Z)$ extends (non-uniquely) to a bivector
field on $\R^N$.

\end{document}